\documentclass[reqno,10pt]{amsart}

\usepackage{style}

\title{A note on the recovery sequence in the double gradient model for phase transitions}

\author[J. Deutsch]{Jakob Deutsch} 
\address[Jakob Deutsch]{TU Wien}
\email{jakob.deutsch@tuwien.ac.at}

\begin{document} 

\begin{abstract}
    We investigate the $\limsup$ inequality in the double gradient model for phase transitions governed by a Modica--Mortola functional with a double-well potential in two dimensions. Specifically, we consider energy functionals of the form
    \[
        E_\varepsilon(u, \Omega) = \int_\Omega \left( \frac{1}{\varepsilon} W(\nabla u) + \varepsilon |\nabla^2 u|^2 \right) dx
    \]
    for maps $ u \in H^2(\Omega; \mathbb{R}^2) $, where $ W $ vanishes only at two wells. Assuming a bound on the optimal profile constant --- namely the cell problem on the unit cube --- in terms of the geodesic distance between the two wells, we characterise the limiting interfacial energy via periodic recovery sequences as $\varepsilon \to 0^+$.
\end{abstract}

\maketitle



\setlength\parindent{12pt}
\section{Introduction}

The study of diffuse interface models for phase transitions has been central to the calculus of variations and mathematical materials science for decades. The seminal Modica--Mortola functional 
\[
    F_\eps(u, \Omega) := \int_{\Omega} \frac{1}{\eps}W(u) + \eps|\nabla u|^2 \, dx,
\]
and its variants provide a classical variational description of liquid-liquid phase separation. The sharp interface limit of such functionals, obtained via $\Gamma$-convergence, yields surface energies that concentrate on $(n-1)$-dimensional interfaces and is now a standard cornerstone of the field \cite{Modica1987, ModicaMortola1977, Gurtin1987}. There is a vast literature connecting singularly perturbed energies to geometric variational problems that has inspired numerous generalisations, among which we mention some to the vectorial \cite{FonsecaTartar1989, Baldo1990, Bouchitte1990, CristoferiGravina2021} and higher-order settings \cite{FonsecaMantegazza2000,   ChermisiDalMasoFonsecaLeoni2011, BruscaDonatiSolci2025}. 

Considering analogous variational models dealing with phase transformations in solids, the deformations are vectorial, and thus the relevant variables are matrix valued. A prototypical model in this direction is the double-gradient Modica--Mortola functional
\[
    E_\varepsilon(u,\Omega)=\int_\Omega\!\Big(\frac{1}{\varepsilon}\,W(\nabla u)+\varepsilon\,|\nabla^2 u|^2\Big)\,dx, 
\]
with a simply connected Lipschitz domain $\Omega \subset \R^2$, a deformation $u \colon \Omega \to \R^2$, a non-negative potential $W \colon \mathbb{R}^{2\times2}\to[0,\infty)$ vanishing on a finite set of possibly set-valued wells, and a parameter $\eps > 0$ representing the thickness of transitional interfaces. The authors {\sc S. Conti}, {\sc I. Fonseca} and {\sc G. Leoni} \cite{ContiFonsecaLeoni2002} established a $\Gamma$-convergence result for this class of functionals and derived a cell formula based on periodic functions for the limiting interfacial energy density for two single point wells. While the $\Gamma$-$\liminf$ inequality was proven in a very general setting, the construction of the recovery sequence requires making certain restrictive assumptions on the potential. We remark that there has also been an extensive effort to incorporate frame-indifference into the potential. In the celebrated results of \cite{ContiSchweizer2006lin, ContiSchweizer2006nonlin, ContiSchweizer2006nonlinimp} the $\Gamma$-limit was computed in two dimensions (cf.\ also \cite{StinsonKerrek2021}). The problem in higher dimensions in its full generality remains open, although significant advances have recently been made in \cite{DavoliFriedrich2020, DavoliFriedrich2025}. 

A central object in studying the $\Gamma$-limit is the optimal profile constant, which is computed by taking the $\Gamma$-$\liminf$ on the unit cube with respect to a single jump interface, i.e., 
\begin{align}\label{intro_optimal_profile_constant}
    K^* = \Gamma(L^1) \text{-} \liminf_{\eps \to 0} E_{\eps}(u_0, (1/2, 1/2)^2)
\end{align}
where the interface normal is taken to be $e_2$ and 
\[
    u_0(x) = \begin{cases}
        A x & \text{if } x_2 \geq 0, \\
        B x & \text{if } x_2 < 0,
    \end{cases}
\]
for two wells $A = -B = a \otimes e_2$ and $a \in \R^2$ as a simplification. This constant appears as the surface energy density in the $\Gamma$-limit and governs the cost per unit of an interface separating two wells of an optimal configuration in the limiting space $BV(\Omega; \{ A, B \})$. In practice, this formulation is too abstract to employ in actual problems. One would thus like to express it in terms of a simpler structure. This has proven to be considerably difficult and has only been achieved in the setting of two single point wells where certain symmetry assumptions are placed on the potential $W$, see \cite{ContiFonsecaLeoni2002}. We note that in the case of fluid-fluid phase transitions, the optimal profile constant \eqref{intro_optimal_profile_constant} reduces to a one-dimensional geodesic problem (cf.\ \cite{FonsecaTartar1989}).

The main contribution of this paper is the extension of the result by \cite{ContiFonsecaLeoni2002} in two dimensions under the assumption of a certain bound on the optimal profile constant in terms of the geodesic distance of the two wells. More precisely, we introduce an a priori bound in terms of a geodesic distance between the wells $A$ and $B$ (cf.\ Definition \ref{def:geodesic_distance})
\begin{align}\label{assumption_bound_on_opt_prof_const}
    K^*<3\,d_W(A,B).
\end{align}
We prove that, if \eqref{assumption_bound_on_opt_prof_const} holds in conjunction with standard regularity and (quadratic) growth assumptions for the potential $W$, then $K^*$ can be expressed as a minimisation problem over periodic gradients on the unit cube. More precisely, we show that
\begin{align}\label{intro_equ}
    K^* = \inf\Bigg\{ &\int_{(-1/2, 1/2)^2} L W(\nabla u) + \frac{1}{L}|\nabla^2 u|^2 \, dx : \, L > 0, \,  u \in H^2((-1/2, 1/2)^2; \R^2), \,  \nonumber \\ 
    & \nabla u \text{ is 1-periodic in } x_1,\, \nabla u(x) = \nabla u_0(x) \text{ for $x_1 \in (-1/2, 1/2),\, |x_2| \in (1/4, 1/2)$}
        \Bigg\}
\end{align}
holds. This characterisation of the optimal profile constant $K^*$ has already been suggested in \cite{ContiFonsecaLeoni2002} under the symmetry condition 
\[
    W(m_1, m_2) = W(-m_1, m_2).
\]
Under this assumption, the equality \eqref{intro_equ} can be shown, essentially, by reflecting and glueing optimal profiles. Although this approach is certainly viable in the presence of symmetry, it fails in its absence. We circumvent this difficulty by introducing a novel approach of glueing together optimal profiles in two dimensions based on a careful analysis of their traces. We show that the bound on the optimal profile constant translates to a bound on the energy of suitable traces of carefully selected one-dimensional line segments. This bound, in turn, implies that said traces are separated in exactly two connected regions such that the gradient along these lines is close to the wells in these regions. This makes an optimal glueing easier.

The paper is organised as follows. In Section 2, we give an overview of the preliminaries needed. More precisely, in Subsection \ref{sec:notation} we introduce the notation used throughout this paper. In Subsection \ref{sec:discussion_on_geodesic_distances}, we recall certain important properties of the geodesic distance and curves fulfilling the geodesic length bound of $3d_W(A,B)$. The main contribution of the paper is presented in Section 3. This section begins with a discussion on modifications of optimal profiles in Subsection \ref{sec:horizontal_modification}, where we gather certain statements from \cite{ContiFonsecaLeoni2002} and improve them in a suitable way. Subsection \ref{sec:optimal_interpolation_of_traces} is devoted to a glueing procedure of curves which are subject to certain energy bounds. This is, in a sense, the heart of this paper: We introduce a method here which will allow us to create transitional maps to pass from one trace to another. In Subsection \ref{sec:inequality_proof}, we prove our main result Theorem \ref{theorem:main_result}. Lastly, in Subsection \ref{sec:geodesic_distance_bound} we will introduce a class of potentials $W$ which fulfil the assumption \eqref{assumption_bound_on_opt_prof_const}. This class is essentially comprised of all perturbations (in a suitable sense) of quadratic double well potentials.






\setlength\parindent{0pt}

\section{Setting and Preliminaries}
\subsection{Notation \& Setup}\label{sec:notation}
In the following, we denote the unit cubes in 1D and 2D with $q = (-1/2, 1/2)$ resp. $Q = (-1/2, 1/2)^2$. We will use $C$ for a generic changing positive constant when writing estimates. We denote important dependencies on parameters by subscripts ($C_\sigma$, $C_\tau$, etc.). For a matrix $M \in \R^{2 \times 2}$, we sometimes write $(m_1, m_2)$ where $m_i$ is the $i$-th column. Furthermore, for $r > 0$ and a set $M \subset \R^2$ we write $B_r(M) = \{ x \in \R^2 : \dist(M, x) < r \}$. We use the usual notation $B_r(x)$ for open balls at a point $x \in \R^2$ with radius $r > 0$. For $h > 0$ and $\omega \subset \R$ we define the cylinders $\omega_h := \omega \times (-h, h)$. Furthermore, we set $\omega_* := \omega_{1/2} = \omega \times (-1/2, 1/2)$. We consider the Modica--Mortola type functional 
\[
    E_\eps(u,\Omega) := \int_\Omega \frac{1}{\eps}W(\nabla u) + \eps |\nabla^2 u|^2 \, dx
\]
for $u \in H^2(\Omega, \R^2)$, $W \colon \R^{2 \times 2} \to \R_{\geq 0}$, $\eps > 0$ and a Lipschitz domain $\Omega \subset \R^2$. We select two matrices $A,B \in \R^{2 \times 2}$ which we will call wells. Similarly to the model in \cite{ContiFonsecaLeoni2002}, we restrict ourselves to $$A = -B = a \otimes e_2 = (0, a)$$ with $a \in \R^2 \setminus \{0\}$ and $e_2$ being the second unit vector. We further denote for $(m_1 , m_2) = M \in \R^{2 \times 2}$
\[
    W_0(M) := |m_1|^2 + \min |m_2 \pm a|^2  = \min\{ |M - A|^2, |M - B|^2 \}.
\]
For brevity, we use notations of the type $\min|M \pm A|^2 := \min\{ |M - A|^2, |M - B|^2 \} $. Generally, we use $\pm$ as a placeholder for the two cases with $+$ and $-$.
We assume that $W$ is a double well potential, i.e., for all $M \in \R^{2 \times 2}$ we have $W(M) = 0$ if and only if $M \in \{ A, B \}$. Moreover, we impose the following conditions on the potential $W$:
\begin{enumerate}[label = (H\arabic *) ]
    \item\label{assump:well_and_potential_assumption} $W$ is continuous.
    \item\label{assump:quadratic_growth_around_wells} We assume that $W$ has global quadratic growth around the wells, i.e., there exist $C > 0$ such that for $(m_1 , m_2) = M \in \R^{2 \times 2}$
    \[
        \frac{1}{C}W_0(M) \leq W(M) \leq CW_0(M).
    \]
\end{enumerate}
Throughout the paper, for all statements we will assume \ref{assump:well_and_potential_assumption} and \ref{assump:quadratic_growth_around_wells}.

\begin{remark}\label{remark:gathering_of_facts_in_the_beginning}
    It has been shown in \cite[Remark 6.1]{ContiFonsecaLeoni2002} that the combination of \ref{assump:well_and_potential_assumption} and \ref{assump:quadratic_growth_around_wells} implies the existence of a constant $C > 0$ such that for all $M, N \in \R^{2\times 2}$
    \begin{align}\label{assump:quadratic_lipschitz}
        W(M) \leq C(W(N) + |M - N|^2).
    \end{align}
    Moreover, \ref{assump:quadratic_growth_around_wells} implies
    \begin{align}\label{assump:quadratic_from_below_first_variable}
        |m_1|^2 \leq CW(M)
    \end{align}
    for all $(m_1, m_2) = M \in \R^{2\times 2}$. Furthermore, we note that, for all $\alpha > 0$, condition \ref{assump:quadratic_growth_around_wells} implies the existence of a constant $C_\alpha > 0$ such that for all $M \in B_\alpha(\{\pm A \})^c$ we have
    \begin{align}\label{assump:inverse_quadratic_condition}
        \max|M \pm A|^2 \leq C_\alpha \min|M \pm A|^2 \leq C_\alpha W(M).
    \end{align}
    Indeed, we observe that 
    \[
        \lim_{|M| \to \infty}\frac{|M + A|}{|M - A|} = 1
    \]
    holds and we have for any $R > 0$ and for $M \in B_{R}(0) \setminus B_\alpha(A)^c$ 
    \[
        \frac{|M + A|}{|M - A|} \leq \frac{R + |A|}{\alpha}.
    \]
    Now, we choose $R > 0$ such that for all $M \in B_R(0)^c$ we have 
    \[
        \frac{|M + A|}{|M - A|} < 2.
    \]
    We consequently have for a large $C > 0$ (independent of $\alpha$) 
    \[
        |M + A| \leq \max\left\{2,\frac{C}{\alpha} \right\} |M - A|
    \]
    for all $M \in M \in B_\alpha(\{\pm A \})^c$.
    By symmetry, \eqref{assump:inverse_quadratic_condition} holds with $C_\alpha = \max\left\{2,C/\alpha \right\}$. We note here that for small $\alpha > 0$ we have 
    \[
        C_\alpha = \frac{C}{\alpha}.
    \]
\end{remark}

We now recall the definition of an optimal profile energy:

\begin{definition}\label{def:optimal_profile}
    For $h > 0$ and an open $\omega \subset \R$ we define the optimal profile energy by 
    \begin{align}\label{definition:optimal_profile_energy}
        \F(\omega, h) := \inf\left\{ \liminf_{n \to \infty} E_{\eps_n}(u_n, \omega_h): \eps_n \to 0^+, u_n \to u_0 \text{ in } L^1(\omega_h, \R^2) \right\} 
    \end{align}
    where 
    \[
        u_0(x) := Ax\chi_{\{x_2 > 0\}}(x) + Bx \chi_{\{x_2 \leq 0\}}(x).
    \]
\end{definition}

\begin{remark}\label{remark:vertical_boundary_conditions}
    In \cite[Lemma 4.3]{ContiFonsecaLeoni2002} it has been shown that $\F$ is independent of $h$ and that if \ref{assump:well_and_potential_assumption} holds, we have
    \[
        \F(\omega) := \F(\omega, h) = \H^1(\omega)K^*,
    \]
    where the constant $K^*$ is defined as $K^* := \F(q)$. Therefore, $\F(\cdot)$ is a multiple of the Hausdorff measure.
    Moreover, if \ref{assump:quadratic_growth_around_wells} holds, we have 
    \begin{align*}
        K^* = \inf\bigg\{ \liminf_{n \to \infty} E_{\eps_n}(u_n, q_h): & ~ \eps_n, c^+_n, c^-_n \to 0^+, u_n \to u_0 \text{ in } L^1(\omega_h, \R^2)  \\ & ~  u_n(x) = u_0(x) + c^\pm_n \quad \text{ for } x \in (\omega \times ((-h, -2/3h) \cup (2/3h, h))) \bigg\}.
    \end{align*} 
    This was shown in \cite[Proposition 6.2]{ContiFonsecaLeoni2002} for $h = 1$. The proof of this proposition can be directly adapted for any $h > 0$. The information about vertical boundary conditions is crucial for glueing.
\end{remark}

Next, we will introduce the concept of an optimal profile sequence: 
\begin{definition} \label{def:opt_prof_constant_and_opt_prof} Let $h > 0$. We say that a pair of sequences $(u_n, \eps_n) \subset  H^2(\omega_h, \R^2) \times (0,1)$ (which are admissible in taking the infimum in \eqref{definition:optimal_profile_energy}) is an optimal profile sequence with respect to $\omega$ if it attains the minimum, i.e., if 
\[
    \lim_{n \to \infty} E_{\eps_n}(u_n, \omega_h) = \H^{n-1}(\omega)K^*.
\]
\end{definition}
Via a diagonalization argument, it can be shown that the infimum in the definition of the optimal profile energy is in fact a minimum. Therefore, the existence of an optimal profile sequence is always guaranteed. One crucial property of optimal profile sequences is the local optimality: 
\begin{lemma}\label{lemma:local_optimality_of_optimal_profiles}
    Let $h > 0$. Suppose $(u_n, \eps_n) \subset  H^2(\omega_h, \R^2) \times (0,1)$ is an optimal profile sequence with respect to $\omega \subset \R$. Then, it is also an optimal profile sequence with respect to any open set $\tilde \omega \subset \omega$  with $|\partial \tilde \omega|  = 0$. 
\end{lemma}
\begin{proof}
    We just note that by Definition \ref{def:optimal_profile}
    \begin{align*}
        \limsup_{n \to \infty} E_{\eps_n}(u_n, \tilde \omega_h) & = \lim_{n \to \infty} E_{\eps_n}(u_n, \omega_h) - \liminf_{n \to \infty} E_{\eps_n}(u_n, (\omega \setminus \tilde \omega)_h)  \\
        & \leq \F(\omega) - \F(\omega \setminus \overline{\tilde \omega}) = \F(\tilde \omega).
    \end{align*}
    Since, by the definition of $\F$, we have
    \[
        \F(\tilde \omega) \leq  \liminf_{n \to \infty} E_{\eps_n}(u_n, \tilde \omega_h),
    \]
    we conclude with 
    \[
        \F(\tilde \omega) = \lim_{n \to \infty} E_{\eps_n}(u_n, \tilde \omega_h).
    \]
\end{proof}

\subsection{Discussion on geodesic distances}\label{sec:discussion_on_geodesic_distances}

In the theory of liquid-liquid phase transitions, i.e., the classical vectorial Modica--Mortola functional, the geodesic distance with respect to $W$ plays a crucial role in the explicit computation of the $\Gamma$-limit (cf.\ \cite{FonsecaTartar1989}). For the theory of solid-solid phase transitions, the geodesic distance function is generally not used since the $\Gamma$-limit is computed via a cell formula over periodic functions. However, in our analysis we will still use a certain property of curves which have a bound on their geodesic length. Here, we recap the definition and some of the basic properties of geodesic distances. 
\begin{definition}\label{def:geodesic_distance_function}
    Let $I$ be any closed interval and $\varphi \in W^{1,1}(I; \R^{2 \times 2})$. We call 
    \[
        L_W(\varphi) := 2\int_I \sqrt{W(\varphi(s))}|\varphi'(s)| \, ds 
    \]
    the length of $\varphi$ with respect to $W$. We further denote the geodesic distance with respect to $W$ between the two matrices $M, N \in \R^{2 \times 2}$ by 
    \begin{align}\label{def:geodesic_distance}
        d_W(M,N) := \inf\{ L_W(\varphi) : \, \varphi \in W^{1,1}(I; \R^{2 \times 2}), \, \varphi(-1) = M, \varphi(1) = N\}.
    \end{align}
    To simplify the analysis that follows, we will use $I = [-1,1]$ and note that the above quantities are invariant under reparametrization of $I$.
\end{definition}
\begin{lemma}\label{lemma:geodesic_distance_lipschitz}
    The geodesic distance function $d_W \colon \R^{2 \times 2} \times \R^{2 \times 2} \to \R_{\geq 0}$ defined in \eqref{def:geodesic_distance} is locally Lipschitz continuous. 
\end{lemma}
\begin{proof}
    Let $R > 0$ and $A, \tilde A, B \in B_R(0)$. Let  $\varphi \in W^{1,1}([-1,1]; \R^{2 \times 2}), \, \varphi(-1) = A, \varphi(1) = B$ such that for $\eps > 0$ we have
    \[
        L_W(\varphi) \leq d_W(A,B) + \eps.
    \]
    Let $\zeta \colon [-1, 1] \to \R^{2 \times 2}$ and $\tilde \zeta \colon [-1,0] \to \R^{2 \times 2}$ be the linear interpolation of $\tilde A$ and $A$ on the respective intervals, $\tilde \varphi \colon [0,1] \to \R^{2 \times 2}$ be the reparametrization of $\varphi$ defined by $\tilde \varphi(s) := \varphi(-1 + 2s)$ and $\psi \colon [-1,1] \to \R^{2 \times 2}$ defined by
    \[
        \psi(s) := \begin{cases}
            \tilde \zeta, & \text{in } [-1, 0], \\
            \tilde \varphi, & \text{in } [0, 1].
        \end{cases}
    \]
    Now, since 
    \[
        \L_W(\psi) = L_W(\zeta) + L_W(\varphi)
    \]
    we can derive from the definition of the geodesic distance
    \begin{align*}
        d_W(\tilde A, B) - d_W(A, B) &\leq L_W(\psi) - L_W(\varphi) + \eps \\
        &= L_W(\zeta) + \eps \\
        &\leq C\sup_{M \in B_R(0)} \sqrt{W(M)} |A - \tilde A| + \eps.
    \end{align*}
    Since $\eps > 0$ was arbitrary and by symmetry, we infer for any $\tilde A, A, B \in B_R(0) \times B_R(0)$
    \[
        |d_W(\tilde A, B) - d_W(A, B)| \leq C\sup_{M \in B_R(0)} \sqrt{W(M)} |A - \tilde A|
    \]
    By \ref{assump:well_and_potential_assumption} we know that $\sup_{M \in B_R(0)} \sqrt{W(M)}$ is finite. We have shown that $d_W$ is Lipschitz in the first variable in $B_R(0) \times B_R(0)$. By the symmetry of the geodesic distance function, $d_W$ is also Lipschitz in the second variable in $B_R(0) \times B_R(0)$. From this we can conclude that $d_W$ is locally Lipschitz.
\end{proof}
To motivate condition \eqref{assumption_bound_on_opt_prof_const} we make the following observation:
\begin{lemma}\label{lemma:existence_of_phase_regions}
    Let $\varphi \in W^{1,1}([-1,1]; \R^{2 \times 2})$ with $\varphi(-1) = M \in \R^{2\times 2}, \varphi(1) = N \in \R^{2\times 2}$ that fulfils $L_W(\varphi) < 3d_W(M,N)$. Then, for each $\alpha > 0$ with 
    \begin{align}\label{eq:admissibility_constant}
        \alpha < \gamma(M, N, W, \varphi) := \min \left\{ \frac{|M- N|}{2}, \frac{3d_W(M,N) - L_W(\varphi)}{8L}\right\}
    \end{align}
    where $L > 0$ is the Lipschitz constant of $d_W$ on the region $B_R(0) \times B_R(0)$ with $R = 2\max\{|M|, |N|\}$ we have
    \[
        \sup \varphi^{-1}(B_\alpha(M)) < \inf \varphi^{-1}(B_\alpha(N)). 
    \]
    
\end{lemma}
\begin{proof}
    We first assume $\alpha < |M - N|/2$. Note that this implies that $\sup \varphi^{-1}(B_\alpha(M)) = \inf \varphi^{-1}(B_\alpha(N))$ is not possible since $\varphi$ is continuous. We prove the statement by contradiction. Assume that the reverse inequality
    \[
        \sup \varphi^{-1}(B_\alpha(M)) >\inf \varphi^{-1}(B_\alpha(N))
    \]
    holds.
    If this is the case, we could find  $s_M^\alpha \in \varphi^{-1}(B_\alpha(M))$ and $s_N^\alpha \in \varphi^{-1}(B_\alpha(N))$ with $s_N^\alpha < s_M^\alpha$. By Lemma \ref{lemma:geodesic_distance_lipschitz} $d_W$ is Lipschitz continuous on $(B_R(M) \cup B_R(N)) \times (B_R(M) \cup B_R(N))$ with Lipschitz constant $L > 0$. By our choice of $\alpha$ we have $\varphi(s_M^\alpha), \varphi(s_N^\alpha) \in B_R(0)$. We derive
    \[
        |d_W(M,N) - d_W(\varphi(s_M^\alpha)), \varphi(s_N^\alpha))| \leq L(|M - \varphi(s_M^\alpha)| + |N - \varphi(s_N^\alpha)|) \leq 2L\alpha.
    \]
    Analogously, we have
    \[
        |d_W(M,N) - d_W(M, \varphi(s_N^\alpha))| \leq L\alpha
    \]
    and 
    \[
        |d_W(M,N) - d_W(\varphi(s_M^\alpha), N)| \leq L\alpha.
    \]
    From these estimates, we can derive 
    \begin{align*}
        3d_W(M,N) \leq d_W(M, \varphi(s_N^\alpha)) + d_W(\varphi(s_N^\alpha), \varphi(s_M^\alpha)) + d_W(\varphi(s_M^\alpha), N) + 4L\alpha
        \leq L_W(\varphi) + 4L\alpha.
    \end{align*} 
    So for 
    \begin{align}
        \alpha < \frac{3d_W(M,N) - L_W(\varphi)}{8L}
    \end{align}
    we derive a contradiction.
\end{proof}

Now, we introduce the concept of an admissible curve. These are essentially those curves where we transition from one phase to another only once which means we can set a point which separates the phases, i.e., above this point we are close to one phase and below we are close to the other one.

\begin{definition}\label{def:admissable_curves}
    Let $\varphi \in W^{1,1}([-1,1]; \R^{2\times 2})$ with $\varphi(-1) = M, \varphi(1) = N$.  Then, we call the pair $(\varphi, \alpha)$ admissible if $\varphi$ fulfils the assumptions of Lemma \ref{lemma:existence_of_phase_regions} and $\alpha < \gamma(M,N,W,\varphi)$ where $\gamma$ is explicitly defined in \eqref{eq:admissibility_constant}. Moreover, we define the phase separating point for an admissible pair $(\varphi, \alpha)$ by 
    \[
        s_\varphi^\alpha := \frac{1}{2}(\sup \varphi^{-1}(B_\alpha(M)) + \inf \varphi^{-1}(B_\alpha(N))).
    \]
\end{definition}

Next, we show that admissible curves enjoy certain properties.

\begin{lemma}\label{lemma:estimates_for_the_difference_and_discussion_for_alpha}
    Let $\varphi, \psi \in W^{1,1}([-1,1], \R^{2 \times 2})$ with equal endpoints $\varphi(-1) = \psi(-1) = A$ and $\varphi(1) = \psi(1) = B$.  Then, the following holds:
    \begin{enumerate}[label=(\roman*)]
        \item Suppose that $(\varphi, \alpha)$ and $(\psi, \alpha)$ are admissible pairs for some $\alpha > 0$ and the phase separating points coincide $s_\varphi^\alpha = s_\psi^\alpha$. Then, we have for every $s \in [-1,1]$ the estimate
        \[
            |\varphi(s) - \psi(s)|^2 \leq C_\alpha (W(\varphi(s)) + W(\psi(s))).
        \]
        \item Let $K > 0$. Suppose that $\varphi$ fulfils
        \[
            L_W(\varphi) < K < 3d_W(A,B).
        \]
        Then,  $(\varphi, \alpha_K)$ is admissible for 
        \begin{align}\label{eq:definition_alpha_K}
            \alpha_K := \min\{(3d_W(A,B) - K)/(12L), |A-B|/8\}.
        \end{align}
        Here, $L > 0$ is the Lipschitz constant of $d_W$ on the region $B_R(0)\times B_R(0)$ with $R = 2\max\{|A|, |B|\}$.
    \end{enumerate}
\end{lemma}

\begin{proof}
    We start by proving $(i)$. We have for $s \in \varphi^{-1}(B_\alpha(A))$ and $\alpha < R$ 
    \[
        |\varphi(s) - \psi(s)|^2 \leq C(|\varphi(s) - A|^2 + |\psi(s) - A|^2) \leq (W(\varphi(s)) + |\psi(s) - A|^2).
    \]
    If $\psi(s) \in B_\alpha(B)$ we would have by the definition of the phase separating points 
    \[
        s_\psi^\alpha > \sup \psi^{-1}(B_\alpha(B)) > s > s_\varphi^\alpha
    \]
    which contradicts our assumption. So we must have $\psi(s) \in B_\alpha(A)$ or $\psi(s) \in B_\alpha(\{A,B\})^c$. In the first case, we apply \ref{assump:quadratic_growth_around_wells} and, in the second case, \eqref{assump:inverse_quadratic_condition} to derive
    \[
        |\psi(s) - A|^2 \leq C_\alpha W(\psi(s)).
    \]
    Analogous estimates can be made for $s \in \varphi^{-1}(B_\alpha(B))$ and $s \in \varphi^{-1}(B_\alpha(\{A,B\})^c)$. In conclusion, we derive for every $s \in [-1,1]$ the pointwise estimate
    \begin{align*}
        |\varphi(s) - \psi(s)|^2 \leq C_\alpha(W(\varphi(s)) + W(\psi(s))).
    \end{align*}
    
    For $(ii)$, we just observe that
    \[
        \alpha_K < \gamma(A,B, W, \varphi)
    \]
    by comparing $\alpha_K$ to the definition of $\gamma$ (cf.\ \eqref{eq:admissibility_constant}). By definition, this means that $(\varphi, \alpha_K)$ is admissible.
\end{proof}

\section{Main Result}
We are now in a position to state the main result of this paper. We recall the definition of the optimal profile constant $K^*$ from Definition \ref{def:opt_prof_constant_and_opt_prof}, and of the geodesic distance $d_W$ from Definition \ref{def:geodesic_distance_function}.
\begin{theorem}\label{theorem:main_result}
    Suppose that $\ref{assump:well_and_potential_assumption}$, $\ref{assump:quadratic_growth_around_wells}$ and $$K^* < 3d_W(A,B)$$ 
    holds. Then,
    \begin{align*}
        K^* = K^*_\rm{per} := \inf\Bigg\{ &\int_Q L W(\nabla u) + \frac{1}{L}|\nabla^2 u|^2 \, dx : \, L > 0, \,  u \in H^2(Q; \R^2), \, \nabla u \text{ is 1-periodic in } x_1, \\ 
        & \nabla u(x) = \nabla u_0(x) \text{ for $x_1 \in (-1/2, 1/2),\, |x_2| \in (1/4, 1/2)$}
        \Bigg\},
    \end{align*}
    where $Q$ is the unit cube $(-1/2, 1/2)^2$ and $u_0$ is as in Definition \ref{def:optimal_profile}.
\end{theorem}
It was already shown in the proof of \cite[Proposition 6.4]{ContiFonsecaLeoni2002} that under \ref{assump:well_and_potential_assumption} we have $K^* \leq K^*_\rm{per}$. We highlight that only the continuity of $W$ is needed so that one can apply the Riemann-Lebesgue Lemma. For convenience, we recall this fact with a sketch of its proof.
\begin{proposition}[{\cite[Proposition 6.4]{ContiFonsecaLeoni2002}}]
    Suppose that $\ref{assump:well_and_potential_assumption}$ holds. Then, we have 
    \[
        K^* \leq K^*_\rm{per}.
    \]
\end{proposition}
\begin{proof}
    Let $\delta > 0$. First, let $u \in H^2(Q;\R^2)$ which almost minimizes $K^*_\rm{per}$, i.e., there exists an $L > 0$ such that
    \[
        K^*_\rm{per} + \delta > \int_Q L W(\nabla u) + \frac{1}{L}|\nabla^2 u|^2 \, dx.
    \]
    Then, for any sequence of positive numbers $\{\eps_n\}$ with $\eps_n \to 0$ we define a rescaled sequence of maps $z_{\eps_n} \in H^2(Q;\R^2)$ (cf.\ \cite[(6.20)]{ContiFonsecaLeoni2002}) such that 
    \[
        \nabla z_{\eps_n}(x) = \begin{cases}
            a \otimes e_2 & \text{if } x_2 > \frac{\eps L}{2} \\
            \nabla u\left( \frac{x}{\eps L} \right) & \text{if } |x_2| \leq \frac{\eps L}{2} \\
            - a \otimes e_2 & \text{if } x_2 < -\frac{\eps L}{2}.
        \end{cases}
    \]
    Here, we used the fact that we can extend $\nabla u$ periodically in $x'$. We note that 
    \[
        E_{\eps_n}(z_{\eps_n}) = \int_Q L W\left( \nabla u\left( \frac{x'}{\eps L}, t \right) \right) + \frac{1}{L}\left|\nabla^2 u\left( \frac{x'}{\eps L}, t \right)\right|^2 \, d(x',t) .
    \]
    Now, we apply the Riemann Lebesgue Lemma to infer 
    \[
        \liminf_{n \to \infty} E_{\eps_n}(z_{\eps_n}) = \int_Q L W(\nabla u) + \frac{1}{L}|\nabla^2 u|^2 \, dx
    \]
    and, consequently,
    \[
        K^* \leq \liminf_{n \to \infty} E_{\eps_n}(z_{\eps_n})  \leq K^*_\rm{per} + \delta.
    \]
    Since $\delta >0$ was arbitrary we conclude.
\end{proof}

For completeness, we also recall that the equality $K^* = K^*_\rm{per}$ paired with \ref{assump:well_and_potential_assumption} and \ref{assump:quadratic_growth_around_wells} implies 
\[
    \Gamma(L^1)-\lim_{n \to \infty} E_{\eps_n}(u, \Omega) = K^*_\rm{per}\H^{1}(J_{\nabla u} \cap \Omega)
\]
(cf.\ \cite[Theorem 6.6 \& 6.7]{ContiFonsecaLeoni2002}) for any bounded, simply connected Lipschitz domain $\Omega \subset \R^2$. \\[0,6em]
The remainder of this section is devoted to showing the inequality $K^* \geq K^*_\rm{per}$. 
We will start the discussion with Lemma \ref{lemma:horizontal_boundary_condition}, where we show that optimal profile sequences can be modified suitably at the horizontal boundary. This idea was originally used in \cite{ContiFonsecaLeoni2002}, but we refined it here, so that the trace of the gradient satisfies a suitable energy bound. Afterwards, in Lemmas \ref{lemma:trace_translation_glueing} -- \ref{lemma:trace_glueing_if_midpoints_close}, we are concerned with glueing together traces with certain energy bounds. In Theorem \ref{thm:horizontal_modification}, we combine these auxiliary steps to derive $K^*_\rm{per} \leq K^*$. Lastly, in Theorem \ref{thm:quadratic_perturbations_fulfil_requirement} we will show that the assumptions of the main result are fulfilled as long as the potential $W$ is a suitable perturbation of a quadratic potential.  \\[0,6em]

\subsection{Horizontal modification of optimal profiles}\label{sec:horizontal_modification}

We start off the discussion with a recap of the vertical modification of an optimal profile found in \cite{ContiFonsecaLeoni2002} (cf.\ also the comments in Remark \ref{remark:vertical_boundary_conditions}):
\begin{proposition}{\cite[Proposition 6.2]{ContiFonsecaLeoni2002}}\label{theorem:vertical_boundary_condition}
    Let $h > 0$ and $(u_n, \eps_n) \subset H^2(Q, \R^2) \times (0,1)$ be an optimal profile sequence with respect to $q = (-1/2, 1/2)$ in the sense of Definition \ref{def:opt_prof_constant_and_opt_prof}. Then, there exists an optimal profile sequence $(w_n, \eps_n) \subset H^2(Q, \R^2) \times (0,1)$ with respect to $q$ and null sequences $c_n^\pm \in \R^2$ such that 
    \[
        w_n(x) = 
        \begin{cases}
            x_2 a + c^+_n & \text{ for }x_2 > \frac{2 h}3, \\
            -x_2 a + c^-_n & \text{ for } x_2 < -\frac{2 h}3.
        \end{cases}
    \]
\end{proposition}

Recall, that we use the notation 
\[
    \omega_* = \omega \times (-1/2, 1/2)
\]
to denote cylinders with $\omega \subset \R$. We will now present our modification of an optimal profile at the boundary:

\begin{figure}
    \centering
    \includegraphics[width=0.75\linewidth]{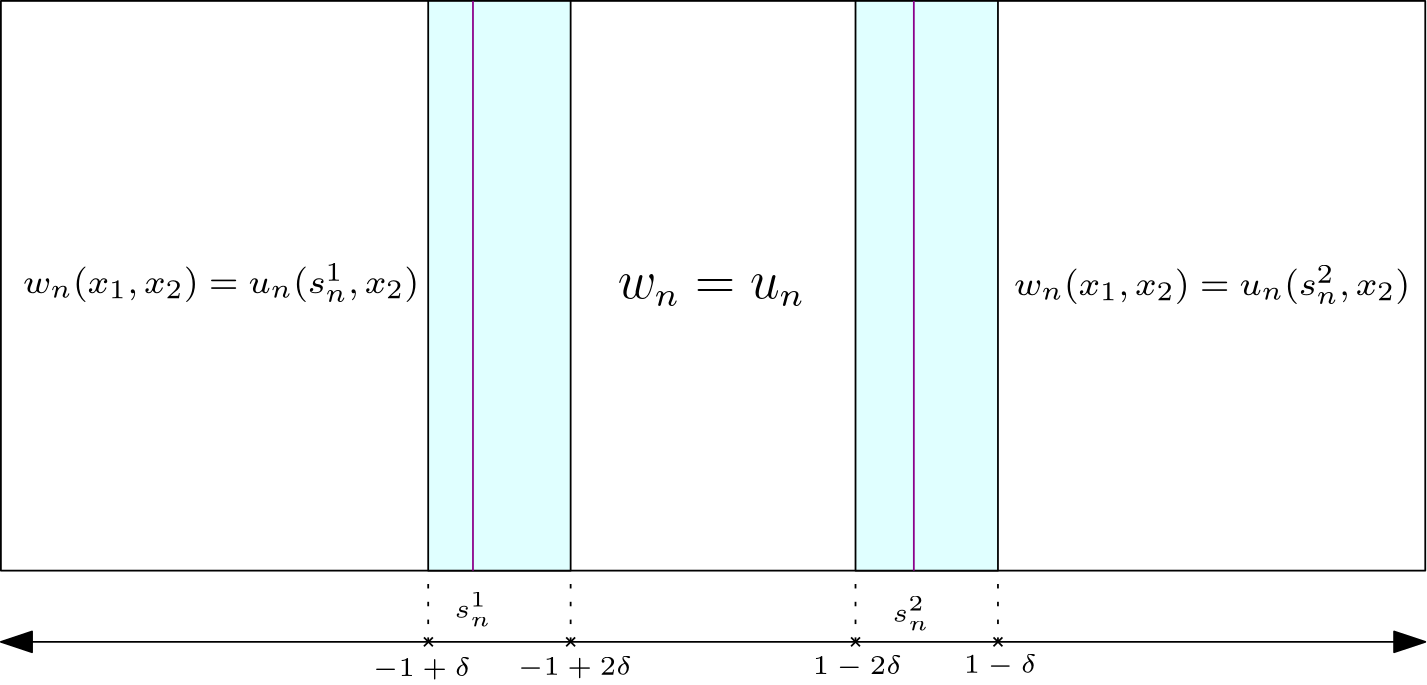}
    \caption{The modification from Lemma \ref{lemma:horizontal_boundary_condition}. A De Giorgi type argument allows one to choose suitable $s_n^1, s_n^2$ such that the energy along the trace of $u_n$ is bounded along ${s^i_n} \times (-1/2, 1/2)$ in a suitable way whilst keeping the energy of the modification $w_n$ in the light blue area proportional to $\delta$.}
    \label{fig:horizontal_boundary_condition_illustration}
\end{figure}

\begin{lemma}\label{lemma:horizontal_boundary_condition}
    Let $\delta \in (0, 1/4)$, $\tau \in (1, 2)$, and let $(u_n, \eps_n) \subset H^2(Q, \R^2) \times (0,1)$ be an optimal profile sequence with respect to $q$. Then, there exists $C_\tau > 0$ and $n_0 \in \N$ such that for every $n \geq n_0$ there exists a sequence $(w_n) \subset H^2((-1/2 - \delta, 1/2 + \delta))$ with the following properties:
    \begin{itemize}
        \item We have $w_n = u_n$ in $(-1/2 + 2\delta, 1/2 - 2\delta)_*$.
        \item There exists $s_n^1 \in (- 1/2 + 2\delta, - 1/2 + \delta)$ and $s_n^2 \in (1/2 - 2\delta, 1/2 - \delta)$ such that $w_n$ admits the trace values of $u_n$ in the boundary regions: For all $x_2 \in (-1/2, 1/2)$ we have
        \begin{itemize}
            \item $w_n(x_1, x_2) = u_n(s_n^1, x_2) $ for all $x_1 \in (-1/2 - \delta, -1/2 + \delta)$, and
            \item $w_n(x_1, x_2) = u_n(s_n^2, x_2) $ for all $x_1 \in (1/2 - \delta, 1/2 + \delta)$.
        \end{itemize}
        \item The following energy estimate along the trace holds: For every $i \in \{1,2\}$ we have
        \begin{align}\label{eq:trace_energy_condition}
            \int_{-\frac 1 2}^{\frac 1 2} \frac 1 {\eps_n} W\left((\nabla u_n)(s_n^i, x_2)\right) + \eps_n  |\nabla^2 u_n(s_n^i, x_2)|^2 \, dx_2 < \tau K^*.
        \end{align}
        \item We have
        \[
            \limsup_{n \to \infty} E_{\eps_n}(w_n, ((-1/2 + 2\delta, -1/2 + \delta)\cup(1/2 - 2\delta, 1/2 - \delta))_*) \leq C_{\tau}\delta.
        \]
        
    \end{itemize}
    As a consequence,
    \[
        \limsup_{n \to \infty} E_{\eps_n}(w_n, (-1/2 - \delta, 1/2 + \delta)_*) \leq K^*(1 + C_{\tau}\delta).
    \]
    See Figure \ref{fig:horizontal_boundary_condition_illustration} for illustration.
\end{lemma}


\begin{proof}
    We will only treat the extension of $u_n$ to the right side of $Q$, as the other side can be discussed analogously. By Lemma \ref{lemma:local_optimality_of_optimal_profiles}, we have 
    \[
        \lim_{n \to \infty} E_{\eps_n}(u_n, (1/2 - \delta, 1/2)_*) = K^*\delta.
    \]
    Let $\tilde \tau \in (1,\tau)$. We infer 
    \[
        E_{\eps_n}(u_n, (1/2 - \delta, 1/2)_*) < \tilde \tau K^* \delta
    \]
    for $n$ large enough. We divide the interval $(1/2 - 2\delta, 1/2 - \delta)$ into $m_n := \lfloor \frac 1 {\eps_n} \rfloor$ subintervals 
        $$I_k := \left(1/2 - 2\delta + \frac{k-1}{m_n}\delta, \, 1/2 - 2\delta + \frac{k}{m_n}\delta \right) $$ 
    with $k = 1,.., m_n$. We observe that  
    \[
        \sum_{k = 1}^{m_n} E_{\eps_n}(u_n, (I_k)_*) < \tilde \tau K^* \delta.
    \]
    On the other hand, for large $n$ we also obtain
    \[
        \sum_{k = 2}^{m_n} \int_{(I_{k} \cup I_{k-1})_*}|\nabla u_n - \nabla u_0|^2 + |u_n - u_0| \, dx < \delta^3K^* \frac{(\tau-\tilde \tau)(\tilde \tau - 1)}{2\tilde \tau}
    \]
    since $u_n \to u_0$ in $H^1(Q, \R^2)$. Using Lemma \ref{lemma:combinatorics_lemma} in the appendix, we find $k_0 \in \{2, .., m_n\}$ with 
    \[
        E_{\eps_n}(u_n, (I_{k_0})_*) \leq \frac{\tilde \tau \delta}{m_n} K^*,
    \]
    \[
        E_{\eps_n}(u_n, (I_{k_0} \cup I_{k_0 - 1})_*) \leq \frac{4\tilde \tau}{(\tilde \tau - 1)}\frac{\delta}{m_n}  K^*
    \]
    and 
    \[
        \int_{(I_{k_0} \cup I_{k_0 - 1})_* } |\nabla u_n - \nabla u_0|^2 + |u_n - u_0| \, dx < \frac{(\tau - \tilde \tau ) \delta^3}{m_n} K^*.
    \]
    This implies
    \[
        \int_{(I_{k_0})_*} \frac 1 {\eps_n} W(\nabla u_n) + \eps_n |\nabla^2 u_n|^2 + \frac 1 {\delta^2} |\nabla u_n - \nabla u_0|^2 + |u_n - u_0|  \, dx < \frac{\tau \delta}{m_n}K^*
    \]
    and
    \begin{align}\label{eq:trace_estimate_1}
        E_{\eps_n}(u_n, I_{k_0} \cup I_{k_0 - 1}) +  \int_{(I_{k_0} \cup I_{k_0 - 1})_*}\frac 1 {\delta^2} |\nabla u_n - \nabla u_0|^2 + |u_n - u_0|  \, dx < \frac{C_\tau \delta}{m_n}K^*.
    \end{align}
    In particular, we can choose $s_n \in I_{k_0}$ such that
    \begin{align}\label{eq:trace_estimate_2}
        \int_{(-\frac 1 2, \frac 1 2)} \frac 1 {\eps_n} W(\nabla u_n(s_n, t)) + \eps_n |\nabla^2 u_n(s_n, t)|^2 + \frac 1 {\delta^2} |(\nabla u_n - \nabla u_0)(s_n, t)|^2 + |(u_n - u_0)(s_n, t)|  \, dt < \tau K^*.
    \end{align}
    Let $\varphi_{n} \in C^\infty(\R; [0,1])$ such that $\varphi_{n}(s) = 1$ for $s \geq s_n$, $\varphi_{n}(s) = 0$ for $s \leq s_n - \delta/m_n$, and
    \[
        |\varphi_n'| \leq \frac{Cm_n}{\delta}  \qquad |\varphi_n''| \leq \frac {Cm_n^2}{\delta^2}. 
    \]
    Now, define 
    \[
        w_{n}(x) := \varphi_{n}(x_1)u_n(s_n, x_2) + (1 - \varphi_{n}(x_1))u_n(x)
    \]
    on $Q$. By construction the first two requirements on $w_n$ are fulfilled.
    The main work, namely showing that 
    \begin{align}\label{trace_estimate_3}
         E_{\eps_n}(w_n, (1/2 - 2\delta, 1/2)_*) \leq C_\tau\delta.
    \end{align}
    holds, has already been done in Theorem 6.3 in \cite{ContiFonsecaLeoni2002} where (6.9) and (6.11) in \cite{ContiFonsecaLeoni2002} corresponds to our $\eqref{eq:trace_estimate_1}$ and $\eqref{eq:trace_estimate_2}$ estimates from which \eqref{trace_estimate_3} is derived. A simplified version of this fact can be found in the appendix, see Lemma \ref{app:estimate_of_interpolation}.
    Note now that extending $w_n$ to $(-1/2, 1/2 + \delta)_*$ by setting $w_n(x_1, x_2) := u_n(s_n, x_2)$ does only change the energy proportionally to $\delta$ since
    \begin{align*}
        \int_{(\frac 1 2, \frac 1 2 + \delta)_*} \frac 1 {\eps_n} W(\nabla w_n(x)) \, dx &= \delta \int_{(-\frac 1 2, \frac 1 2)} \frac 1 {\eps_n} W((0, \partial_2 u_n(s_n, t)) \, dt \\
        &\leq C\delta \left( \frac 1 {\eps_n}\int_{(-\frac 1 2, \frac 1 2)} |\partial_1 u_n(s_n, t)|^2 + W((\nabla u_n)(s_n, t)) \, dt \right) \\
        &\leq C\delta.
    \end{align*}
    Here, we employed \eqref{assump:quadratic_lipschitz}, \eqref{assump:quadratic_from_below_first_variable} and \eqref{eq:trace_estimate_2}. Moreover, we observe 
    \[
        \nabla^2 w_n(x) = \begin{pmatrix}
             0 & 0 \\ 0 & \partial_{22} u_n(s_n, x_2 )
        \end{pmatrix}.
    \]
    Consequently, using again \eqref{eq:trace_estimate_2}, we derive
    \[
        E_\eps(w_n, (1/2, 1/2 + \delta)_*) \leq C\delta
    \]
    which concludes the proof.
\end{proof}

\subsection{Optimal interpolation of traces}\label{sec:optimal_interpolation_of_traces}

For $\eps > 0$ we define the energy of a curve $\varphi \in  W^{1,1}_{\rm{loc}}(\R; \R^{2 \times 2})$ by 
\[
    I_\eps(\varphi) := \int_\R \frac{1}{\eps}W(\varphi) + \eps |\varphi'|^2 \, ds.
\]
We note that by applying the standard Modica--Mortola trick (Young inequality) we have 
\[
    L_W(\varphi) \leq I_\eps(\varphi).
\]
where $L_W$ is given in Definition \ref{def:geodesic_distance}. Now, we will begin the discussion about the procedure to glue two traces together in an optimal way. For this, we will derive energy bounds of certain transitional maps. We start with an observation on vertical translations of traces. 
\begin{lemma}\label{lemma:trace_translation_glueing}
    Let $\varphi \in  W^{1,1}_{\rm{loc}}(\R; \R^{2 \times 2})$ and $\eps \in (0,1)$ with $I_{\eps}(\varphi) < \infty$. Then, for each $\sigma > 0$ and $\beta \in (-1,1)$ there exists a map $\tilde v \in H^2((-\sigma, \sigma)_*, \R^2)$ and a constant $C_\sigma > 0$ with the following properties:
    \begin{enumerate}[label = (\roman*)]
        \item $\partial_2 \tilde v(x_1, x_2 ) = \varphi_2(x_2)$ near $x_1 = -\sigma$ for all $x_2$,
        \item $\partial_2 \tilde v(x_1, x_2) = \varphi_2(x_2 + \beta)$ near $x_1 = \sigma$ for all $x_2$,
        \item $\partial_1 \tilde v = 0$ near $x_1 = \pm \sigma$,
        \item We have the energy estimate
        \[
            E_{\eps}(\tilde v, (-\sigma, \sigma)_*) \leq C \left( \sigma + C_\sigma \frac{\beta^2}{\eps} \right)(1 + I_\eps(\varphi)) .
        \]
    \end{enumerate}
\end{lemma}
\begin{proof}
    Let $\beta \in (-1,1)$ and $\sigma > 0$. We define
    \[
        \tilde v(x_1, x_2) := \int_{-1/2}^{x_2 + \rho(x_1)} \varphi_2(s) \, ds 
    \]
    for $(x_1, x_2) \in (-\sigma, \sigma)_*$, where $\rho \in C^\infty(-\sigma,\sigma)$ with 
    \begin{itemize}
        \item $\rho(-\sigma) = 0$, $\rho(\sigma) = \beta$, and $\rho' = 0$ near $\pm \sigma$,
        \item $|\rho'| \leq C\beta/\sigma$ and $|\rho''| \leq C\beta/\sigma^2$.
    \end{itemize}
    Notice that 
    \[
        \partial_2 \tilde v(x_1, x_2) = \varphi_2(x_2 + \rho(x_1))
    \]
    and 
    \[
        \partial_{22} \tilde v(x_1, x_2) = \varphi_2'(x_2 + \rho(x_1)).
    \]
    Therefore, by the choice of $\rho$ we know that $\tilde v$ fulfils (i) and (ii). By \ref{assump:quadratic_growth_around_wells} we have
    \begin{align*}
        \int_{(-\sigma, \sigma)_*} \frac{1}{\eps}\min\{ |\partial_2 \tilde v \pm a|^2 \} + \eps |\partial_{22} \tilde v|^2 \, dx &\leq \int_{(-\sigma, \sigma)_*} \frac{1}{\eps}\min\{ |\varphi_2(x_2 + \rho(x_1)) \pm a| \} + \eps |\varphi_2'(x_2 + \rho(x_1))|^2 \, dx\\ 
        &\leq \sigma \int_{\R} \frac{1}{\eps}\min\{ |\varphi_2(x_2) \pm a| \} + \eps |\varphi_2'(x_2)|^2 \, dx\\
        &= \sigma I_\eps(\varphi).
    \end{align*}
    For the remaining partial derivatives, we find
    \[
        \partial_1 \tilde v(x_1, x_2) = \varphi_2(x_2 + \rho(x_1))\rho'(x_1),
    \]
    \[
        \partial_{11} \tilde v(x_1, x_2) = \varphi_2'(x_2 + \rho(x_1))(\rho'(x_1))^2 + \varphi_2(x_2 + \rho(x_1))\rho''(x_1),
    \]
    as well as
    \[
        \partial_{12} \tilde v(x_1, x_2) = \varphi_2'(x_2 + \rho(x_1))\rho(x_1).
    \]
    We now estimate the first derivative 
    \begin{align*}
        \int_{-\sigma}^\sigma \int_{-1/2}^{1/2} \frac{1}{\eps}| \partial_1 \tilde v|^2 \, dx_2 \, dx_1 
        &\leq \frac{C\beta^2}{\sigma^2}\int_{-\sigma}^\sigma \int_\R \frac{1}{\eps} |\varphi_2|^2 \, dx_2 \, dx_1 \\ 
        &\leq \frac{C\beta^2}{\sigma^2}\int_{-\sigma}^\sigma \int_\R \frac{1}{\eps} (\min|\varphi_2 \pm a|^2 + 1) \, dx_2 \, dx_1  \\
        &\leq \frac{C\beta^2}{\sigma}\left(I_{\eps}(\varphi) + \frac{1}{\eps} \right),
    \end{align*}
    and the remaining second derivatives where we repeatedly use $|\rho^{k}| \leq \beta/(\sigma^k)$ for $k \in \{1,2\}$, $|\beta| \leq 1$, \eqref{assump:quadratic_lipschitz} and \eqref{assump:quadratic_from_below_first_variable}
    \begin{align*}
        \int_{-\sigma}^\sigma \int_{-1/2}^{1/2} \eps| \partial_{11} \tilde v|^2 \, dx_2 \, dx_1 
        &\leq C\int_{-\sigma}^\sigma \int_\R  |\varphi_2'(x_2 + \rho(x_1))(\rho'(x_1))^2|^2 + |\varphi_2(x_2 + \rho(x_1))\rho''(x_1)|^2 \, dx_2 \, dx_1\\
        &\leq \frac{C\eps (\beta^4 + \beta^2)}{\sigma^4} \int_{-\sigma}^\sigma \int_\R  |\varphi_2'|^2 + |\varphi_2|^2 \, dx_2 \, dx_1\\
        &\leq \frac{C\beta^2}{\sigma^3}\left( I_\eps(\varphi) + \eps \right), 
    \end{align*}
    and 
    \begin{align*}
        \int_{-\sigma}^\sigma \int_{-1/2}^{1/2} \eps| \partial_{12} \tilde v|^2 \, dx_2 \, dx_1 &\leq  \frac{C\eps \beta^2}{\sigma^2} \int_{-\sigma}^\sigma \int_\R  |\varphi_2'|^2 \, dx_2 \, dx_1 \\
        & \leq \frac{C\beta^2}{\sigma} I_\eps(\varphi).
    \end{align*}
    Putting these estimates together, we derive 
    \begin{align*}
        E_\eps(\tilde v, (-\sigma, \sigma)_*) &\leq C\int_{(-\sigma, \sigma)_*} \frac{1}{\eps}(|\partial_1 \tilde v|^2 + \min |\partial_2 \tilde v \pm a|^2) + \eps|\nabla^2 \tilde v|^2 \, dx \\
        &\leq C \left(C_\sigma \left(\frac{\beta^2}{\eps} +  I_\eps(\varphi) \right) + \sigma I_\eps(\varphi) \right)  \\
        &\leq  C \left( \sigma + C_\sigma \frac{\beta^2}{\eps} \right)(1 + I_\eps(\varphi)) . \\
    \end{align*}
    This concludes the proof.
\end{proof}
In the next lemma, we will observe that, as long as phase separating points (cf.\ Definition \ref{def:admissable_curves}) coincide, we can transition between traces of gradients such that the energy of the transitional map is controlled in terms of the width of the transitional region and the height of the interface.
\begin{lemma}\label{lemma:trace_the_same_midpoint_glueing}
    Let $\varphi, \psi \in  W^{1,1}_{\rm{loc}}(\R, \R^{2\times 2})$, $K > 0$ and $\eps \in (0,1)$  with $I_\eps(\varphi), I_{\eps}(\psi) < K < 3d_W(A,B)$. 
    Furthermore, let $s_\varphi = s_\varphi^{\alpha_K}, \, s_\psi = s_\psi^{\alpha_K}$ as in Lemma \ref{lemma:estimates_for_the_difference_and_discussion_for_alpha}. Suppose that 
    \begin{itemize}
        \item $s_\varphi = s_\psi$, and
        \item there exists $h> \eps >0 $ such that $\varphi(s) = \psi(s) = \pm A$ for $\pm s \geq \pm h$. 
    \end{itemize} 
    Then, for each $\sigma > \eps > 0$ there exists a map $\tilde w = \tilde w_{\sigma} \in H^2((-\sigma, \sigma)_*, \R^2)$ with the following properties:
    \begin{enumerate}
        \item $\partial_2 \tilde w(x_1, .) = \varphi_2$ near $x_1 = -\sigma$,
        \item $\partial_2 \tilde w(x_1, .) = \psi_2$ near $x_1 = \sigma$,
        \item $\partial_1 \tilde w$ = 0 near $x_1 = \pm \sigma$,
        \item We have the energy estimate:
        \[
            E_\eps(\tilde w) \leq C_K\left(\frac{h}{\sigma} + \sigma\right)(I_\eps(\varphi) + I_\eps(\psi)).
        \]
    \end{enumerate}
\end{lemma}
\begin{proof}
    To start with, we note that $s_\varphi$ and $s_\psi$ are well defined since by the Modica--Mortola trick 
    \[
        L_W(\varphi) \leq I_\eps(\varphi) < K \leq 3d_W(A,B)
    \]
    holds.
    By Lemma \ref{lemma:estimates_for_the_difference_and_discussion_for_alpha} $(ii)$ we know that $(\varphi, \alpha_K)$ is admissible in the sense of Definition \ref{def:admissable_curves}, i.e., $s_\varphi^{\alpha_K}$ is well defined (and analogously $s_\psi^{\alpha_K}$).
    Now, we directly define 
    \[
        \tilde w(x_1, x_2) = \tilde\rho(x_1)\int_{-1/2}^{x_2} \varphi_2(s) \, ds + (1-\tilde\rho(x_1))\int_{-1/2}^{x_2} \psi_2(s) \, ds,
    \]
    where $\tilde\rho \in C^\infty((-\sigma,\sigma); [0,1])$ with $\tilde\rho(-\sigma) = 0$, $\tilde\rho(\sigma) =1$ and $\tilde \rho' = 0$ near $\pm \sigma$, $|\tilde\rho'| \leq C/\sigma$, and $|\tilde\rho''| \leq C/\sigma^2$.
    We notice that 
    \[
        \partial_1 \tilde w(x_1, x_2) = \tilde \rho'(x_1)\left(\int_{-1/2}^{x_2} \varphi_2(s) - \psi_2(s) \, ds \right),
    \]
    and 
    \[
        \partial_2 \tilde w(x_1, x_2) = \tilde\rho(x_1) \varphi_2(x_2) + (1-\tilde\rho(x_1)) \psi_2(x_2).
    \]
    Using the quadratic growth assumption \ref{assump:quadratic_growth_around_wells}, we see that
    \[
        \int_{-\sigma}^\sigma \int_{-1/2}^{1/2} W(\nabla \tilde w) \, dx_2 \, dx_1 \leq C\left(\int_{-\sigma}^\sigma \int_{-1/2}^{1/2} |\partial_1 \tilde w|^2 + \min|\partial_2 \tilde w \pm a|^2 \, dx_2 \, dx_1 \right).
    \]
    We estimate the two terms separately. For the first term, we apply Hölder's inequality to derive
    \begin{align}\label{eq:poincare_application}
        \int_{(-\sigma, \sigma)_*} |\partial_1 \tilde w|^2 \, dx \leq \frac{1}{\sigma} \left(\int_{-h}^{h} |\varphi_2(s) - \psi_2(s)| \, ds \right)^2 \leq \frac{h}{\sigma} \int_{-h}^{h} |\varphi_2 - \psi_2|^2 \, ds.
    \end{align}
    Since $s_\varphi = s_\psi$ we can apply Lemma \ref{lemma:estimates_for_the_difference_and_discussion_for_alpha} $(i)$ to derive 
    \begin{align*}
        \int_{-h}^{h} |\varphi_2 - \psi_2|^2 \, ds \leq C_K\eps(I_\eps(\varphi) + I_\eps(\psi)).
    \end{align*}
    The second term is estimated in a similar spirit:
    \begin{align*}
        \int_{-1/2}^{1/2} \min|\partial_2 \tilde w \pm a|^2 \, dx  & = \int_{-1/2}^{1/2} \min|\tilde\rho(x_1) \varphi_2(x_2) + (1-\tilde\rho(x_1)) \psi_2(x_2) \pm a|^2 \, dx \\
        &\leq C\int_{-1/2}^{1/2} \min|\psi_2 \pm a|^2 + |\psi_2 - \varphi_2|^2 \, dx \\
        &\leq C_K\eps(I_\eps(\varphi) + I_\eps(\psi)). \\
    \end{align*}
    Therefore, we observe
    \begin{align*}
        \int_{-\sigma}^\sigma \int_{-1/2}^{1/2} \min|\partial_2 \tilde w \pm a|^2 \, dx_2 \, dx_1 \leq C_K\sigma \eps (I_\eps(\varphi) + I_\eps(\psi)),
    \end{align*}
    and, consequently,
    \begin{align}\label{eq:trace_the_same_midpoint_glueing_1}
        \frac{1}{\eps}\int_{(-\sigma, \sigma)_*} W(\nabla \tilde w) \, dx \leq C_K\left(\frac{h}{\sigma} + \sigma \right)(I_\eps(\varphi) + I_\eps(\psi))
    \end{align}
    Now, we compute the second derivatives and get 
    \[
        \partial_{11} \tilde w(x_1, x_2) =   \tilde \rho''(x_1)\left(\int_{-1/2}^{x_2} \varphi_2(s) - \psi_2(s) \, ds \right)
    \]
    \[
        \partial_{12} \tilde w(x_1, x_2) =   \tilde \rho'(x_1)(\varphi_2(x_2) - \psi_2(x_2))
    \]
    and
    \[
        \partial_{22} \tilde w(x_1, x_2) = \tilde\rho(x_1) \varphi_2'(x_2) + (1-\tilde\rho(x_1)) \psi_2'(x_2).
    \]
    Similar to \eqref{eq:poincare_application} we have
    \begin{align*}
        \eps\int_{(-\sigma, \sigma)_*} |\partial_{11} \tilde w(x_1, x_2)|^2 \, dx &\leq \frac{C\eps h}{\sigma^3} \int_{-1/2}^{1/2}|\varphi_2 - \psi_2|^2 \, ds
    \end{align*}
    and, therefore, we infer 
    \[
        \eps\int_{(-\sigma, \sigma)_*} |\partial_{11} \tilde w(x_1, x_2)|^2 \, dx \leq C_K \frac{\eps^2 h}{\sigma^3}(I_\eps(\varphi) + I_\eps(\psi)).
    \]
    Moreover,
    \begin{align*}
        \eps\int_{(-\sigma, \sigma)_*} |\partial_{12} \tilde w(x_1, x_2)|^2 \, dx \leq \frac{C\eps}{\sigma} \int_{-h}^h |\varphi_2 - \psi_2|^2 \, ds \leq C_K \frac{\eps^2}{\sigma}(I_\eps(\varphi) + I_\eps(\psi))
    \end{align*}
    and
    \begin{align*}
        \eps\int_{(-\sigma, \sigma)_*} |\partial_{22} \tilde w(x_1, x_2)|^2 \, dx \leq C\sigma \eps \left( \int_{-h}^h |\varphi_2'|^2 \, ds +  \int_{-h}^h |\psi_2'|^2 \, ds \right) \leq \sigma (I_\eps(\varphi) + I_\eps(\psi)).
    \end{align*}
    By assumption, we have $\sigma > \eps$ and $h > \eps$, so the inequalities for the second derivatives can be simplified to
    \[
        \eps\int_{(-\sigma, \sigma)_*} |\nabla^2 \tilde w|^2 \, dx \leq C\left( \frac{h}{\sigma} + \sigma \right)(I_\eps(\varphi) + I_\eps(\psi)).
    \]
    Considering this with \eqref{eq:trace_the_same_midpoint_glueing_1} concludes the proof.  
\end{proof}

In the next lemma, we will combine the two previous results to derive a transitional map between any two suitable traces, provided their phase-separating points are not too far apart.
\begin{lemma}\label{lemma:trace_glueing_if_midpoints_close}
    Let $\varphi, \psi \in W^{1,1}_{\rm{loc}}(\R; \R^{2 \times 2})$ and $K > 0$ with  
    \[
        I_{\eps}(\varphi), I_{\eps}(\psi) < K < 3d_W(A,B).
    \]
    Suppose that
    \begin{itemize}
        \item there exists $h \in (\eps, 1)$ such that $\varphi(s) = \psi(s) = \pm A$ for $s \geq h$,
        \item there exists $\tilde h \in (0,1)$ with $h > \tilde h \sqrt \eps > 0$ such that $|s_\varphi - s_\psi| < \tilde h \sqrt \eps < 1$.
    \end{itemize} 
    Then, for each $\sigma > 0$ there exists a map $\tilde z \in H^2((-\sigma, \sigma)_*; \R^2) $ such that 
    \begin{enumerate}
        \item $\partial_2 \tilde z(x_1, .) = \varphi_2$ near $x_1 = -\sigma$,
        \item $\partial_2 \tilde z(x_1, .) = \psi_2$ near $x_1 = \sigma$,
        \item $\partial_1 \tilde z$ = 0 near $x_1 = \pm \sigma$,
        \item we have the energy estimate
        \[
            E_{\eps}(\tilde z, (-\sigma, \sigma)_*) \leq C(\sigma + C_\sigma(h + \tilde h))(1 + I_\eps(\varphi) +     I_\eps(\psi)).
        \]
    \end{enumerate}
\end{lemma}
\begin{proof}
    Consider the curves 
    \[
        \varphi(s), \, \psi(s), \, \zeta(s) := \varphi(s - s_\varphi + s_\psi)
    \]
    for $s \in \R$. First, apply Lemma \ref{lemma:trace_translation_glueing} to $\varphi_2(s)$ and $\beta = - s_\varphi + s_\psi$ to derive a map $\tilde v \in H^2((-\sigma, 0)_*, \R^2)$ (after rescaling and translating to $(-\sigma, 0)$) such that 
    \begin{enumerate}[label = (\roman*)]
        \item $\partial_2 \tilde v(x_1, x_2) = \varphi_2(x_2)$ near $x_1 = -\sigma$ and for all $x_2$,
        \item $\partial_2 \tilde v(x_1, x_2) = \zeta_2(x_2)$ near $x_1 = 0$ and for all $x_2$,
        \item $\partial_1 \tilde v = 0$ near $x_1 = -\sigma$ and $x_1 = 0$,
        \item  and we have
        \[
            E_{\eps}(\tilde v) \leq C \left( \sigma + C_\sigma \frac{\beta^2}{\eps} \right) I_\eps(\varphi) \leq C \left( \sigma + C_\sigma \tilde h \right)(1 + I_\eps(\varphi)).
        \]
    \end{enumerate}
    Now, we observe that $s_\zeta = s_\psi$ by construction. Furthermore, observe that $\zeta(s) = \psi(s) = \pm A$ holds for $\pm s \geq \pm 2h \geq \pm (h + |\beta|)$. We can therefore apply Lemma \ref{lemma:trace_the_same_midpoint_glueing} to $\zeta$ and $\psi$ (instead of $h$ we use $2h$ \red{discuss}) to find a map $\tilde w \in H^2((0, \sigma)_*, \R^2)$ (after rescaling and translating to $(0, \sigma)$) such that 
    \begin{enumerate}
        \item $\partial_2 \tilde w(x_1, x_2) = \zeta_2(x_2)$ near $x_1 = 0$ and for all $x_2 $,
        \item $\partial_2 \tilde w(x_1, x_2) = \psi_2(x_2)$ near $x_1 = \sigma$ and for all $x_2$,
        \item $\partial_1 \tilde w$ = 0 near $x_1 = 0$ and $x_1 = \sigma$,
        \item and we have
        \[
            E_\eps(\tilde w) \leq C\left(\frac{h}{\sigma} + \sigma\right) (I_\eps(\varphi) + I_\eps(\psi)).
        \]
    \end{enumerate}
    Since the derivatives of $\tilde v$ and $\tilde w$ agree in a small neighborhood around $x_1 = 0$, we can, after adding a constant to $\tilde w$, conclude that 
    \[
        \tilde z := \begin{cases}
            \tilde v, & \text{in } (-\sigma, 0)_* \\
            \tilde w, & \text{in } (0, \sigma)_*
        \end{cases}
    \]
    belongs to $H^2((-\sigma, \sigma)_*; \R^2)$ and satisfies the appropriate trace conditions and energy estimates.
\end{proof}
We now turn to an estimate that enables the application of the glueing procedure described above. It essentially states that, under fixed affine boundary conditions, the phase-separating point of the derivative of a curve cannot deviate significantly, provided that its energy is bounded. We remark here that in the next theorem $\varphi, \psi$ are curves valued in $\R^2$ instead of $\R^{2\times2}$ which was the setting of the previous lemmas.
\begin{lemma}\label{lemma:midpoint_estimate}
    Let $\varphi, \psi \in H^2_{\rm{loc}}(\R; \R^{2})$, $\gamma_\varphi, \gamma_\psi \in H^1_{\rm{loc}}(\R; \R^{2}) $ and $K > 0$. We set 
    \[
        \zeta_\varphi := (\gamma_\varphi, \varphi') \text{ and } \zeta_\psi := (\gamma_\psi, \psi').
    \]
    
    Suppose that 
    \[
        I_\eps(\zeta_\varphi), I_\eps(\zeta_\psi) < K < 3d_W(A,B).
    \]
    Furthermore, suppose that there exists $h>\eps >0 $ and $c_+, c_- \in \R^2$ such that $\varphi(s) = \psi(s) = \pm a s + c_\pm$ for $\pm s \geq \pm h$. Then, the distance of the phase-separating points $s_{\zeta_\varphi} = s_{\zeta_\varphi}^{\alpha_K}$ and $s_{\zeta_\psi} = s_{\zeta_\psi}^{\alpha_K}$ (cf.\ Definition \ref{def:admissable_curves} and Lemma \ref{lemma:estimates_for_the_difference_and_discussion_for_alpha} $(ii)$) satisfies:
    \[
        |s_{\zeta_\varphi} - s_{\zeta_\psi}| \leq C_K\sqrt{h}\sqrt{\eps}\sqrt{1 + I_\eps(\zeta_\varphi) + I_\eps(\zeta_\psi)}.
    \]
\end{lemma}

\begin{proof}
    Without restriction, we assume $s_\varphi < s_\psi$.  Consider the region of phase overlaps
    \[
        \Landau := \zeta_\varphi^{-1}(B_{\alpha_K}(A)) \cap \zeta_\psi^{-1}(B_{\alpha_K}(B)).
    \]
    We will first show that  
    \begin{align}\label{eq:difference_of_midpoint_estimate_1}
        |s_{\zeta_\varphi} - s_{\zeta_\psi}| \leq C_K(|\Landau| + \eps)
    \end{align}
    holds for small $\eps > 0$. Observe that 
    \[
        (s_{\zeta_\varphi}, s_{\zeta_\psi}) \subset (\Landau \cup F)
    \]
    with 
    \[
        F = \{ s \in \R : |\zeta_\varphi(s) \pm A| > \alpha_K \text{ or } |\zeta_\psi(s) \pm A| > \alpha_K  \}.
    \]
    Since $I_\eps(\zeta_\varphi) + I_\eps(\zeta_\psi) < 2K$ we have 
    \[
        |F| \leq \frac{C}{\alpha_K^2}\int_\R W(\zeta_\psi) + W(\zeta_\varphi) \, ds \leq C\frac{K\eps}{\alpha_K^2}
    \]
    from which \eqref{eq:difference_of_midpoint_estimate_1} follows.
    Now, we observe that for any $s \in \Landau$ we have $\varphi'(s), -\psi(s) \in B_{\alpha_K}(a)$, and, consequently, we infer
    \[
        \varphi'(s) - \psi'(s) \in B_{2\alpha_K}(2a)
    \]
    By the definition of $\alpha_K$ we have $\alpha_K < |a|/8$ (cf.\ \eqref{eq:definition_alpha_K}). Therefore, we can derive
    \begin{align}\label{eq:difference_of_midpoint_estimate_2}
        \frac 1 2 |a| |\Landau| & \leq \left|\int_\Landau \varphi'(s) - \psi'(s)  \, ds \right| .
    \end{align}
    We can further use that $\varphi(s) - \psi(s) = 0$ for $\pm s \geq \pm h$ paired with the fundamental theorem of calculus to infer
    \begin{align*}
        \left|\int_\Landau \varphi'(s) - \psi'(s)  \, ds \right| & = \left|- \int_{\Landau^c} \varphi'(s) - \psi'(s) \, ds \right| \\
        & \leq \int_{\Landau^c \cap {\zeta_\varphi}^{-1}(B_{\alpha_K}(A))} \left| \varphi'(s) - \psi'(s) \right| \, ds + \int_{\Landau^c \cap ({\zeta_\varphi}^{-1}(B_{\alpha_K}(A)))^c} \left|\varphi'(s) - \psi'(s)\right| \, ds.
    \end{align*}
    We estimate only the first term since the second one works analogously. We get
    \begin{align*}
        \int_{\Landau^c \cap {\zeta_\varphi}^{-1}(B_{\alpha_K}(A))} \left| \varphi'(s) - \psi'(s) \right| \, ds &= \int_{\Landau^c \cap {\zeta_\varphi}^{-1}(B_{\alpha_K}(A)) \cap (-h, h)} \left| \varphi'(s) - \psi'(s) \right| \, ds \\
        &\leq \int_{\Landau^c \cap {\zeta_\varphi}^{-1}(B_{\alpha_K}(A)) \cap (-h, h)} \left| \varphi'(s) - a| + |a- \psi'(s) \right| \, ds \\
        &\leq C\sqrt{h} \left( \int_{\Landau^c \cap {\zeta_\varphi}^{-1}(B_{\alpha_K}(A)) \cap (-h, h)} | \varphi'(s) - a|^2 + |a- \psi'(s)|^2 \, ds \right)^{1/2},
    \end{align*}
    where the last inequality is an application of the Hölder's inequality. Observe that $s \in \Landau^c \cap {\zeta_\varphi}^{-1}(B_{\alpha_K}(A))$ implies that either $s \in {\zeta_\psi}^{-1}(B_{\alpha_K}(B))$ or $s \in {\zeta_\psi}^{-1}(B_{\alpha_K}(\{ A, B \})^c)$. In particular, we can apply \eqref{assump:inverse_quadratic_condition} and \ref{assump:quadratic_growth_around_wells} to observe
    \begin{align*}
        &\int_{\Landau^c \cap {\zeta_\varphi}^{-1}(B_{\alpha_K}(A)) \cap (-h, h)} | \varphi'(s) - a|^2 + |a- \psi'(s)|^2 \, ds \\
        \leq &\, \int_{\Landau^c \cap {\zeta_\varphi}^{-1}(B_{\alpha_K}(A)) \cap (-h, h)} | \zeta_\varphi(s) - A|^2 + |A- \zeta_\psi(s)|^2 \, ds \\
        \leq &\, C_K \int_{\Landau^c \cap {\zeta_\varphi}^{-1}(B_{\alpha_K}(A)) \cap (-h, h)} W(\zeta_\varphi) + W(\zeta_\psi) \, ds.
    \end{align*}
    Therefore, we have
    \[
        \int_{\Landau^c \cap {\zeta_\varphi}^{-1}(B_{\alpha_K}(A))} \left| \varphi'(s) - \psi'(s) \right| \, ds \leq C_K\sqrt{h}\sqrt{\eps}\sqrt{I_\eps(\zeta_\varphi) + I_\eps(\zeta_\psi)}.
    \]
    With combining the last estimate with \eqref{eq:difference_of_midpoint_estimate_1}, \eqref{eq:difference_of_midpoint_estimate_2} and the assumption $\eps < h$ we can now estimate 
    \[
        |s_{\zeta_\varphi} - s_{\zeta_\psi}| \leq C_K(|\Landau| + \eps) \leq C_K \left(\sqrt{h}\sqrt{\eps}\sqrt{I_\eps(\zeta_\varphi) + I_\eps(\zeta_\psi)} + \eps\right) \leq C_K\sqrt{h}\sqrt{\eps}\sqrt{1 + I_\eps(\zeta_\varphi) + I_\eps(\zeta_\psi)}.
    \]
    This concludes the proof.
\end{proof}

\subsection{The proof of $K^*_\rm{per} \leq K^*$}\label{sec:inequality_proof}

Now, we will use the observations from the previous subsections to derive the main result of this paper. We modify a deformation $u$ in such a way at the boundary, that the traces do not intersect phase bubbles by an application of Section \ref{sec:horizontal_modification} and then use the maps from the previous section to interpolate between two translated versions of $u$. Figure \ref{fig:interpolation_graphics} gives a visual summary of the interpolation used in the proof of the next theorem.
\begin{figure}
    \centering
    \includegraphics[width=0.6\linewidth]{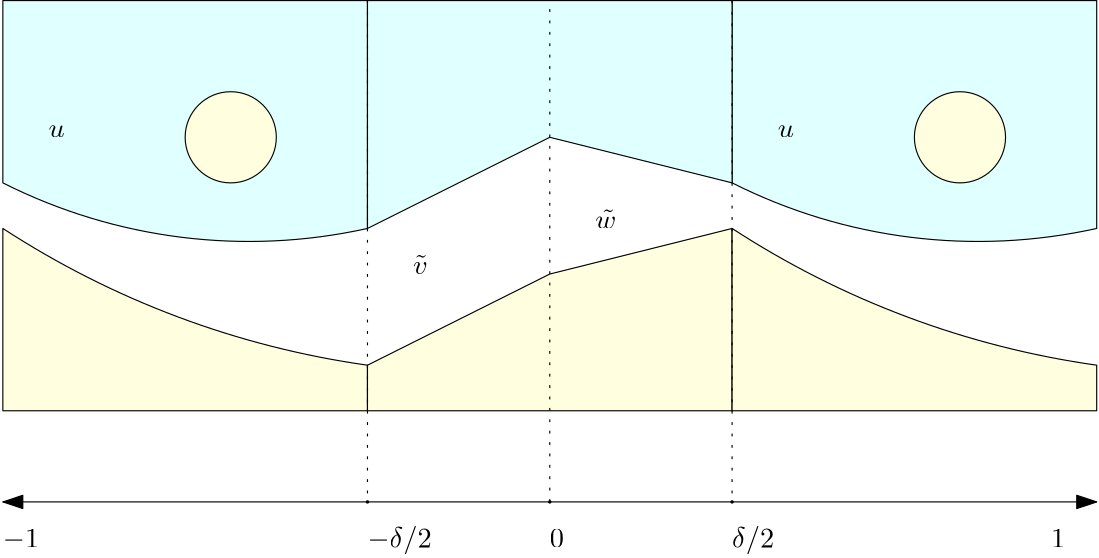}
    \caption{A visualisation of the glueing procedure used in the proof of Theorem \ref{thm:horizontal_modification}. The map $\tilde v$ and $\tilde w$ are given by Lemma \ref{lemma:trace_translation_glueing} and Lemma \ref{lemma:trace_the_same_midpoint_glueing}.}
    \label{fig:interpolation_graphics}
\end{figure}
\begin{theorem}\label{thm:horizontal_modification}
    Suppose that $$K^* < 3d_W(A,B)$$ where $K^*$ is the optimal profile constant (cf.\ Definition \ref{def:opt_prof_constant_and_opt_prof}) and $d_W$ is the geodesic distance from $A, B$ in the sense of \eqref{def:geodesic_distance}. Let $(u_n, \eps_n) \in H^2(Q, \R^2) \times (0,1)$ be an optimal profile sequence with respect to $(-1/2,1/2)$. Then, there exists an optimal profile sequence $(z_n, \eps_n) \in H^2((-1,1)_*, \R^2) \times (0,1)$ with respect to $(-1,1)$ such that the restriction of $\nabla z_n$ to $(-1/2, 1/2)$ is periodic in $x_1$, and we have 
    \[
        K^*_\rm{per} \leq K^*
    \]
    where 
    \begin{align*}
        K^*_\rm{per} = \inf\Bigg\{ &\int_Q L W(\nabla u) + \frac{1}{L}|\nabla^2 u|^2 \, dx : \, L > 0, \,  u \in H^2(Q; \R^2), \, \nabla u \text{ is 1-periodic in } x_1, \\ 
        & \nabla u(x) = \nabla u_0(x) \text{ for $x_1 \in (-1/2, 1/2),\, |x_2| \in (1/4, 1/2)$} \Bigg \}.
    \end{align*}
\end{theorem}
\begin{proof}
    Let $1/2 > \delta > 0$, $\tau \in (1,2)$, $h \in (0, 1)$ and $n$ be large enough such that
    \begin{itemize}
        \item $K^* \tau < 3d_W(A,B)$,
        \item  we have \begin{align}\label{eq:tilde_h_definition}
            C_{K^*\tau}\sqrt{h}\sqrt{1 + 2K^*\tau} =: \tilde h < 1,
        \end{align}
        where $C_{K^*\tau}$ is the constant given by Lemma \ref{lemma:midpoint_estimate}, and
        \item $\eps_n < h/\tilde h$.
    \end{itemize}
    We first modify $(u_n, \eps_n)$ with Proposition \ref{theorem:vertical_boundary_condition} such that (without renaming $u_n$) for $\pm x \geq \pm h$ we have 
    \[
        u_n(x) = \pm ax_2 + c^\pm_n
    \]
    for constants $c^\pm_n \in \R^2$ with $c^\pm_n \to 0$. We apply  Lemma \ref{lemma:horizontal_boundary_condition} with our fixed $\tau$ to find a sequence of vector fields $w_n \in H^2((-1/2 - \delta, 1/2 + \delta)_*, \R^2)$, as well as scalars $s_n^+ \in (1/2 - 2\delta, 1/2 - \delta)$ and $s_n^- \in (-1/2 + 2\delta, -1/2 + \delta)$ such that for large $n \in \N$ the following holds true:
    \begin{itemize}
        \item By our choice of $\tau$ we have $K_\tau := K^* \tau < 3d_W(A,B)$. So property \eqref{eq:trace_energy_condition} in Lemma \ref{lemma:horizontal_boundary_condition} turns into
        \begin{align}\label{eq:trace_relation_to_geodesic}
            \int_{-\frac 1 2}^{\frac 1 2} \frac 1 {\eps_n} W((\nabla u_n)(s, x_2)) + \eps_n  |\nabla^2 u_n(s, x_2)|^2 \, dx_2 < K_\tau < 3d_W(A,B)
        \end{align}
        for $s \in \{s_n^-, s_n^+\}$.
        We remark that \eqref{eq:trace_relation_to_geodesic} is a technical necessity since our trace energy needs to be uniformly bounded away from $3d_W(A,B)$ to apply the results from this section.
        \item $w_n$ admits the trace values of $u_n$ at the boundary: 
        \begin{itemize}
            \item $w_n(x_1, x_2) = u_n(s_n^-, x_2) $ for all $(x_1, x_2) \in (-1/2 - \delta, -1/2 + \delta)_*$,
            \item $w_n(x_1, x_2) = u_n(s_n^+, x_2) $ for all $(x_1, x_2) \in (1/2 - \delta, 1/2 + \delta)_*$.
        \end{itemize}
        \item We have
        $$ w_n = u_n $$
        in $(-1/2 + 2\delta, 1/2 - 2\delta)$.
        \item The following estimate holds: 
        \begin{align}\label{eq:thm_estimate_2}
            E_{\eps_n}(w_n, (-1/2 - \delta, 1/2 + \delta) \times (-1/2, 1/2)) \leq K^*(1 + C\delta).
        \end{align}
    \end{itemize}
    Now, we first set $\tilde \varphi(s) := u_n(s^+_n, s)$ and $\tilde \psi(s) := u_n(s^-_n, s)$, $\gamma_{\tilde\varphi} = \partial_1 u_n(s^+_n, s)$ and $\gamma_{\tilde\psi} = \partial_1 u_n(s^-_n, s)$. Furthermore, we set $\zeta_{\tilde \varphi} = (\gamma_{\tilde \varphi}, \tilde \varphi') = \nabla u_n(s^+_n, s)$ and $\zeta_{\tilde \psi} = (\gamma_{\tilde \psi}, \tilde \psi') = \nabla u_n(s^-_n, s)$. We observe that \eqref{eq:trace_relation_to_geodesic} can be written as
    \[
        I_{\eps_n}(\zeta_{\tilde \varphi}), I_{\eps_n}(\zeta_{\tilde \psi}) \leq K_\tau \leq 3d_W(A,B).
    \]
    Therefore, by Lemma \ref{lemma:estimates_for_the_difference_and_discussion_for_alpha} $(ii)$ the phase separating points $s_{\zeta_{\tilde \varphi}} = s_{\zeta_{\tilde \varphi}}^{\alpha_K}$ and $s_{\zeta_{\tilde \psi}} = s_{\zeta_{\tilde \psi}}^{\alpha_K}$ are well defined. In particular, since we fixed suitable boundary conditions, we can apply Lemma \ref{lemma:midpoint_estimate} and derive
    \[
        |s_{\zeta_{\tilde \varphi} } - s_{\zeta_{\tilde \psi} }| \leq \tilde h\sqrt{\eps_n}.
    \]  
    Now that we have an estimate on the distance of the midpoints, we apply Lemma \ref{lemma:trace_glueing_if_midpoints_close} to $\varphi = \zeta_{\tilde \varphi}$ and $\psi =  \zeta_{\tilde \psi}$, $\tilde h$ as defined in \eqref{eq:tilde_h_definition}, and $\sigma = \delta/2$, to find a map $\tilde z_n \in H^2((-\delta/2, \delta/2)_*; \R^2)$  with
    \begin{align}\label{eq:thm_estimate_1}
        E_\eps(\tilde z_n, (-\delta/2, \delta/2)_*) \leq C_{K_\tau}(\delta + C_\delta\sqrt h),
    \end{align}
    and the property that for all $x_2 \in \R$
    \[
        F(x_1, x_2) := \begin{cases}
            \nabla w_n(x_1 + 1/2, x_2) & x_1 \in (-1, - \delta/2)\\
            \nabla \tilde z_n(x_1, x_2) & x_1 \in (-\delta/2, \delta/2) \\
            \nabla w_n(x_1 - 1/2, x_2) & x_1 \in (\delta/2, 1)
        \end{cases}
    \]
    is in $H^1((-1, 1)_*; \R^2)$ and $\curl$-free by construction. After adding constants $M^1_n$ and $M^2_n$ to $w_n(x_1 - 1/2, x_2)$ and $w_n(x_1 + 1/2, x_2)$ we know that 
    \[
        z_n(x_1, x_2) := \begin{cases}
            w_n(x_1 + 1/2, x_2) + M^1_n & x_1 \in (-1, - \delta/2)\\
            \tilde z_n(x_1, x_2) & x_1 \in (-\delta/2, \delta/2) \\
            w_n(x_1 - 1/2, x_2) + M^2_n & x_1 \in (\delta/2, 1)
        \end{cases}
    \]
    is in $H^2((-1, 1)_*; \R^2)$ with 
    \[
        \limsup_{n \to \infty} E_{\eps_n}(z_n, (-1, 1)_*) \leq \limsup_{n \to \infty} 2E_{\eps_n}(w_n, Q) + \limsup_{n \to \infty}E_{\eps_n}(\tilde z_n, (-\delta/2, \delta/2)_*)  \leq (2 + C\delta + C_\delta \sqrt{h}))K^*,
    \]
    where the latter estimate comes from \eqref{eq:thm_estimate_2} and \eqref{eq:thm_estimate_1}. Note that here the constants $C, C_\delta$ only depend on $K_\tau$. Now, choosing a diagonal sequence letting first $h$ tend to $0$ and then $\delta$, we get an optimal profile sequence $(z_n, \eps_n)$ (without renaming) with respect to $(-1,1)$. By the local optimality property Lemma \ref{lemma:local_optimality_of_optimal_profiles}, we also know that the restriction of $z_n$ to $Q$ is an optimal profile sequence $(z_n|_Q, \eps_n)$ with respect to $(-1/2, 1/2)$. By the construction of $z_n$, we also have the periodicity of $\nabla z_n$ in $x_1$ on $Q$ since 
    \[
        \nabla z_n(-1/2, x_2) = \nabla w_n(0, x_2) = \nabla z_n(1/2, x_2).
    \]
    Finally, we observe
    \[
        K^*_\rm{per} \leq \inf_{n \in \N} E_{\eps_n}(z_n, Q) \leq \limsup_{n \to \infty} E_{\eps_n}(z_n, Q) = K^* 
    \]
    which concludes the proof.
\end{proof}

\subsection{The geodesic distance bound}\label{sec:geodesic_distance_bound}

Here, we show that any continuous perturbation (in a certain sense) of the quadratic potential
\[
    W_0(M) = |m_1|^2 + \min |m_2 \pm a|^2
\]
fulfils the assumptions of Theorem \ref{theorem:main_result}.
\begin{theorem}\label{thm:quadratic_perturbations_fulfil_requirement}
    Suppose that $W$ is continuous and 
    \begin{align}\label{eq:nearly_quadratic}
        |\sqrt W - \sqrt{W_0}| \leq \sigma \sqrt W_0 
    \end{align}
    holds for $\sigma \in (0, 1/2)$. Then, \ref{assump:quadratic_growth_around_wells} and $K^* < 3d_W(A,B)$ hold and $W$ is a double well potential, i.e., $W(M) = 0$ if and only if $M \in \{A,B\}$.
\end{theorem}
\begin{proof}
    We first note that $W(M) = 0$ is equivalent to $W_0(M) = 0$ due to $\sigma < 1$ and assumption \eqref{eq:nearly_quadratic}. Also, \ref{assump:quadratic_growth_around_wells} is a direct consequence, since by \eqref{eq:nearly_quadratic}  we have
    \[
        (1-\sigma)^2W_0 \leq W \leq (1+\sigma)^2W_0.
    \]
    To show $K^* < 3d_W(A,B)$, we set
    \[
        \tilde W(M) := W(0, m_2) \qquad \text{ and } \qquad \tilde W_0(M) := W_0(0, m_2)
    \]
    and
    \[
        K_* = \inf\left\{ \liminf_{n \to \infty} \int_{(-1/2,1/2)} \frac{1}{\eps_n} \tilde W(u_n'(s)) + \eps_n |u_n''(s)|^2 \, ds: \, \eps_n \to 0^+, \, u_n \in H^2(q; \R^2), \,  u_n \to u_0 \text{ in } L^1(q, \R^2) \right\}.
    \]
    Note also that this cell formula has been extensively discussed in \cite[Section 5]{ContiFonsecaLeoni2002}. It has been shown that 
    \[
        K^* \leq K_* = \inf\left\{ \int_{-L}^L \tilde W(g(s)) + |g'(s)|^2 \, ds : L > 0; \, g \in W^{1,1}((-L,L); \R^2), \, g(L) = -g(-L) = a \right\}.
    \]
    By the classical Modica--Mortola reparametrization argument (cf.\ \cite[Proposition 3.2]{Baldo1990} or \cite[Lemma 4.5]{CristoferiGravina2021}) we have 
    \begin{align}
        K_* = d_{\tilde W}(A,B)
    \end{align}
    from which we can deduce 
    \begin{align}\label{eq:inequality_geodesic_dist_cell_formula}
        K^* \leq d_{\tilde W}(A,B).
    \end{align}
    Furthermore, we can observe that
    \begin{align}\label{eq:1D_equals_2D_problem}
        d_{\tilde W_0}(A,B) = d_{W_0}(A,B).
    \end{align}
    holds since $W_0$ fulfils 
    $$
        W_0(M) \geq W_0(0,m_2) = \tilde W_0(M).
    $$ 
    Indeed, take any curve $\varphi \in W^{1,1}([-1,1]; \R^{2 \times 2})$ and notice that 
    \[
        L_{W}((0, \varphi_2)) \leq L_W(\varphi).
    \]
    In particular, we can set the first row of $\varphi$ always to $0$ when taking the infimum over all curves $W^{1,1}([-1,1]; \R^{2 \times 2})$ with $\varphi(-1) = -\varphi(1) = A = a \otimes e_2$ which gives exactly \eqref{eq:1D_equals_2D_problem}. 
    One can also show that since $W_0$ is exactly a quadratic potential, the geodesic from $A$ to $B$ with respect to $W_0$ is just the straight line segment joining $A$ and $B$, which is of the form $(0, ta)$.  Indeed, this can be observed by an application of the co-area formula, which can be found in the proof of \cite[Proposition 3.2]{CristoferiGravina2021}. This also directly implies \eqref{eq:1D_equals_2D_problem}. Now, we just note that for $\varphi \in W^{1,1}([-1, 1], \R^{2})$ we can apply \ref{assump:quadratic_growth_around_wells} to observe
    \[
        \int_{-1}^12\sqrt{W(0, \varphi)}|\varphi'| \, ds \leq (1+ \sigma)\int_{-1}^12\sqrt{W_0(0, \varphi)}|\varphi'| \, ds
    \]
    and for $\psi \in W^{1,1}([-1, 1], \R^{2 \times 2})$
    \[
        \int_{-1}^12\sqrt{W_0(\psi)}|\psi'| \, ds \leq \frac{1}{1 - \sigma} \int_{-1}^12\sqrt{W(\psi)}|\psi'| \, ds.
    \]
    Therefore, we can take the corresponding infimum of the last two inequalities to derive 
    \[
         d_{\tilde W}(A,B) \leq (1 + \sigma)d_{\tilde W_0}(A,B) = (1 + \sigma)d_{W_0}(A,B) \leq \frac{1+ \sigma}{1 - \sigma}d_W(A,B).
    \]
    Since $\sigma < 1/2$ we have 
    \[
        \frac{1+ \sigma}{1 - \sigma} < 3.
    \]
    With this and $\eqref{eq:inequality_geodesic_dist_cell_formula}$ we can conclude 
    \[
        K^* \leq 3d_W(A,B).
    \]
\end{proof}

\appendix

\section{}
Here we will state a short combinatorial lemma:
\begin{lemma}\label{lemma:combinatorics_lemma}
    Suppose that for $n \in \N$ we have finite non-negative sequences $(a_k)_{k = 1}^n, (b_k)_{k = 1}^n$ and $(c_k)_{k = 1}^n$ with the property
    \[
        \sum_{k = 1}^n a_k \leq C_a \qquad \sum_{k = 1}^n b_k \leq C_b \qquad \sum_{k = 1}^n c_k \leq C_c
    \]
    for constants $C_a, C_b$ and $C_c$. Then, for each $\tau > 1$ we find a $k_0 \in \{1, .., n\}$ and $C_\tau > 0$ such that 
    \[
        a_{k_0} \leq \frac{\tau C_a}{n} \qquad b_{k_0} \leq \frac{2\tau C_b}{(\tau - 1)n} \qquad c_{k_0} \leq \frac{2\tau C_c}{(\tau - 1)n}.
    \]
    \begin{proof}
        We just note that 
        \[
            \sum_{k = 1}^n \frac{1}{\tau C_a}a_k + \frac{(\tau - 1)}{2\tau C_b}b_k + \frac{\tau - 1}{2\tau C_c}c_k \leq 1.
        \]
        Therefore, we can find a $k_0 \in \{1 , .., n\}$ such that 
        \[
            \frac{1}{\tau C_a}a_{k_0} + \frac{(\tau - 1)}{2\tau C_b}b_{k_0} + \frac{\tau - 1}{2\tau C_c}c_{k_0} \leq \frac{1}{n}.
        \]
        Since the coefficients are all non-negative the statement follows.
    \end{proof}
\end{lemma}

Here, we present a short auxiliary lemma which is used for modifying optimal profiles at their boundary.

\begin{lemma}\label{app:estimate_of_interpolation}
Let $I = I_0 \cup I_1$ with intervals $I_0 = (0, \delta/m), I_1 = [\delta/m, 2\delta/m)$ for $\delta > 0$ and $m \in \N$. Moreover, let $u \in H^2(I_*, \R^2)$, $\eps > 0$ and $M > 0$ such that there exists $s \in I_0$ with 
\begin{align*}
    \int_{(-\frac 1 2, \frac 1 2)} \frac 1 {\eps} W(\nabla u(s, t)) + \eps |\nabla^2 u(s, t)|^2 + \frac 1 {\delta^2} |(\nabla u - \nabla u_0)(s, t)|^2 + |(u - u_0)(s, t)|  \, dt < M,
\end{align*}
\begin{align*}
    E_{\eps}(u, I) +  \int_{(I_1)_*}\frac 1 {\delta^2} |\nabla u - \nabla u_0|^2 + |u - u_0|  \, dx < \frac{\delta}{m} M,
\end{align*}
and $\varphi \in C^\infty(\R; [0,1])$ with $\varphi(0) = 1$, $\varphi(1) = 0$ and
\[
    |\varphi'| \leq M \frac{\delta}{m} \text{ and } |\varphi''| \leq M\frac{\delta^2}{m^2}.
\]
Then there exists a constant $C_M > 0$ such that the map 
\[
    w(x) = \varphi(x_1)u(x_1,x_2) + (1-\varphi(x_1))u(s, x_2)
\]
fulfils 
\[
    E_\eps(w, I) \leq C_M \left(\delta + \frac{\delta}{\eps m}\right),
\]
and, as a consequence, if $\eps m \geq 1/M$ we have 
\[
    E_\eps(w, I) \leq 2C_M \delta.
\]
\end{lemma}
\begin{proof}
    We first compute the first derivative of $w$:
    \begin{align*}
        \nabla w(x_1, x_2) = (\varphi'(x_1)(u(x_1, x_2) - u(s,x_2)), 0) + \varphi(x_1)\nabla u(x_1,x_2) + (1- \varphi(x_1)) (0, \partial_2 u(s, x_2)).
    \end{align*}
    By our growth assumptions \ref{assump:quadratic_growth_around_wells}, we have  
    \begin{align*}
        W(\nabla w) \leq C\left(1 + \frac{m^2}{\delta^2} |u - \bar u|^2 + |\nabla u - \nabla u_0|^2 + |\overline{\nabla u} - \nabla u_0|^2 \right),
    \end{align*}
    where we have used the short notation  $\overline u(x_1, x_2) =u(s,x_2)$ and $\overline{\nabla u}(x_1, x_2) = \nabla u(s, x_2)$.
    To estimate the difference $u - \bar u$, we apply Poincaré inequality, and we observe
    \begin{align*}
        \int_{I_*} |u - \bar u|^2 \, dx &\leq C\frac{\delta^2}{m^2}\int_{I_*} |\partial_1 (u - \bar u)|^2 \, dx \\
        &\leq  C\frac{\delta^2}{m^2}\int_{I_*} |\nabla (u - u_0)|^2 + |\overline{\nabla u} - \nabla u_0|^2 \, dx \\
        &\leq C_M\frac{\delta^3}{m^3}.
    \end{align*}
    For the other terms, we only need to observe
    \begin{align*}
        \frac{1}{\eps}|I_*| \leq C\frac{\delta}{\eps m}.
    \end{align*}
    By our assumptions and the last two inequalities, we obtain 
    \begin{align*}
        \frac{1}{\eps} \int_{I_*} W(\nabla w) \, dx \leq C_M \frac{\delta}{m\eps}.
    \end{align*}
    Similarly, we can estimate 
    \begin{align*}
        |\nabla^2 w|^2 \leq C\left( |\nabla^2 u|^2 + |\nabla^2 \bar u|^2 + \frac{m^4}{\delta^4}|u - \bar u|^2 + \frac{m^2}{\delta^2}|\nabla u - \nabla \bar u|^2 \right). 
    \end{align*}
    In the exact same fashion as before, the last two terms can be estimated by a Poincaré inequality, with which we can derive 
    \begin{align*}
        \eps \int_{I_*} |\nabla^2 w|^2 \, dx \leq C_M\left( \frac{\delta}{m} + \frac{\delta}{m\eps} \right).
    \end{align*}
    Since the as a consequence' part of the Lemma is an immediate deduction, we conclude.
\end{proof}

\section*{Acknowledgement:} The author gratefully acknowledges Elisa Davoli and Manuel Friedrich for their continued support and for the many valuable discussions during the preparation of this paper. Furthermore, he wishes to thank the inspiring conversations with Kerrek Stinson that helped shape and develop the idea for this paper.

\newpage 
\printbibliography

@article {CristoferiGravina2021,
    AUTHOR = {Cristoferi, Riccardo and Gravina, Giovanni},
     TITLE = {Sharp interface limit of a multi-phase transitions model under
              nonisothermal conditions},
   JOURNAL = {Calc. Var. Partial Differential Equations},
  FJOURNAL = {Calculus of Variations and Partial Differential Equations},
    VOLUME = {60},
      YEAR = {2021},
    NUMBER = {4},
     PAGES = {Paper No. 142, 62},
      ISSN = {0944-2669,1432-0835},
   MRCLASS = {49J45 (26B30 34D15 49J10 76T06)},
  MRNUMBER = {4280521},
       DOI = {10.1007/s00526-021-02008-3},
       URL = {https://doi.org/10.1007/s00526-021-02008-3},
}

@article {ContiFonsecaLeoni2002,
    AUTHOR = {Conti, Sergio and Fonseca, Irene and Leoni, Giovanni},
     TITLE = {A {$\Gamma$}-convergence result for the two-gradient theory of
              phase transitions},
   JOURNAL = {Comm. Pure Appl. Math.},
  FJOURNAL = {Communications on Pure and Applied Mathematics},
    VOLUME = {55},
      YEAR = {2002},
    NUMBER = {7},
     PAGES = {857--936},
      ISSN = {0010-3640,1097-0312},
   MRCLASS = {49J45 (35B25 74G65 74N15)},
  MRNUMBER = {1894158},
MRREVIEWER = {Tom\'a\v s\ Roub\'i\v cek},
       DOI = {10.1002/cpa.10035},
       URL = {https://doi.org/10.1002/cpa.10035},
}

@article {Baldo1990,
    AUTHOR = {Baldo, Sisto},
     TITLE = {Minimal interface criterion for phase transitions in mixtures
              of {C}ahn-{H}illiard fluids},
   JOURNAL = {Ann. Inst. H. Poincar\'e{} C Anal. Non Lin\'eaire},
  FJOURNAL = {Annales de l'Institut Henri Poincar\'e{} C. Analyse Non
              Lin\'eaire},
    VOLUME = {7},
      YEAR = {1990},
    NUMBER = {2},
     PAGES = {67--90},
      ISSN = {0294-1449,1873-1430},
   MRCLASS = {76A99 (49Q05 76E99 76M30 80A15)},
  MRNUMBER = {1051228},
MRREVIEWER = {John\ M.\ Ball},
       DOI = {10.1016/S0294-1449(16)30304-3},
       URL = {https://doi.org/10.1016/S0294-1449(16)30304-3},
}

@article {Modica1987,
    AUTHOR = {Modica, Luciano},
     TITLE = {The gradient theory of phase transitions and the minimal
              interface criterion},
   JOURNAL = {Arch. Rational Mech. Anal.},
  FJOURNAL = {Archive for Rational Mechanics and Analysis},
    VOLUME = {98},
      YEAR = {1987},
    NUMBER = {2},
     PAGES = {123--142},
      ISSN = {0003-9527},
   MRCLASS = {76T05 (80A15)},
  MRNUMBER = {866718},
MRREVIEWER = {L.\ Hsiao},
       DOI = {10.1007/BF00251230},
       URL = {https://doi.org/10.1007/BF00251230},
}

@incollection {Gurtin1987,
    AUTHOR = {Gurtin, Morton E.},
     TITLE = {Some results and conjectures in the gradient theory of phase
              transitions},
 BOOKTITLE = {Metastability and incompletely posed problems ({M}inneapolis,
              {M}inn., 1985)},
    SERIES = {IMA Vol. Math. Appl.},
    VOLUME = {3},
     PAGES = {135--146},
 PUBLISHER = {Springer, New York},
      YEAR = {1987},
      ISBN = {0-387-96462-2},
   MRCLASS = {76T05 (35R35 73B99 80A20)},
  MRNUMBER = {870014},
MRREVIEWER = {C.\ Eugene\ Wayne},
       DOI = {10.1007/978-1-4613-8704-6\_9},
       URL = {https://doi.org/10.1007/978-1-4613-8704-6_9},
}

@article {ModicaMortola1977,
    AUTHOR = {Modica, Luciano and Mortola, Stefano},
     TITLE = {Un esempio di {$\Gamma$}-convergenza},
   JOURNAL = {Boll. Un. Mat. Ital. B (5)},
  FJOURNAL = {Unione Matematica Italiana. Bollettino. B. Serie V},
    VOLUME = {14},
      YEAR = {1977},
    NUMBER = {1},
     PAGES = {285--299},
   MRCLASS = {49A50},
  MRNUMBER = {445362},
MRREVIEWER = {Francesco\ Barbieri},
}

@article {Bouchitte1990,
    AUTHOR = {Bouchitt\'e, Guy},
     TITLE = {Singular perturbations of variational problems arising from a
              two-phase transition model},
   JOURNAL = {Appl. Math. Optim.},
  FJOURNAL = {Applied Mathematics and Optimization},
    VOLUME = {21},
      YEAR = {1990},
    NUMBER = {3},
     PAGES = {289--314},
      ISSN = {0095-4616,1432-0606},
   MRCLASS = {49J45 (35B25 49Q25 80A15)},
  MRNUMBER = {1036589},
MRREVIEWER = {E.\ Acerbi},
       DOI = {10.1007/BF01445167},
       URL = {https://doi.org/10.1007/BF01445167},
}

@article {FonsecaTartar1989,
    AUTHOR = {Fonseca, Irene and Tartar, Luc},
     TITLE = {The gradient theory of phase transitions for systems with two
              potential wells},
   JOURNAL = {Proc. Roy. Soc. Edinburgh Sect. A},
  FJOURNAL = {Proceedings of the Royal Society of Edinburgh. Section A.
              Mathematics},
    VOLUME = {111},
      YEAR = {1989},
    NUMBER = {1-2},
     PAGES = {89--102},
      ISSN = {0308-2105,1473-7124},
   MRCLASS = {49A50},
  MRNUMBER = {985992},
MRREVIEWER = {J.\ E.\ Rubio},
       DOI = {10.1017/S030821050002504X},
       URL = {https://doi.org/10.1017/S030821050002504X},
}

@article {BruscaDonatiSolci2025,
    AUTHOR = {Brusca, Giuseppe Cosma and Donati, Davide and Solci,
              Margherita},
     TITLE = {Higher-order singular perturbation models for phase
              transitions},
   JOURNAL = {SIAM J. Math. Anal.},
  FJOURNAL = {SIAM Journal on Mathematical Analysis},
    VOLUME = {57},
      YEAR = {2025},
    NUMBER = {3},
     PAGES = {3146--3170},
      ISSN = {0036-1410,1095-7154},
   MRCLASS = {49J45 (35B25 35Q56 49J10 74A50)},
  MRNUMBER = {4918627},
       DOI = {10.1137/24M1715325},
       URL = {https://doi.org/10.1137/24M1715325},
}

@article {FonsecaMantegazza2000,
    AUTHOR = {Fonseca, Irene and Mantegazza, Carlo},
     TITLE = {Second order singular perturbation models for phase
              transitions},
   JOURNAL = {SIAM J. Math. Anal.},
  FJOURNAL = {SIAM Journal on Mathematical Analysis},
    VOLUME = {31},
      YEAR = {2000},
    NUMBER = {5},
     PAGES = {1121--1143},
      ISSN = {0036-1410,1095-7154},
   MRCLASS = {49J45 (74B20 74N99)},
  MRNUMBER = {1759200},
MRREVIEWER = {Paolo\ Marcellini},
       DOI = {10.1137/S0036141099356830},
       URL = {https://doi.org/10.1137/S0036141099356830},
}

@article {ChermisiDalMasoFonsecaLeoni2011,
    AUTHOR = {Chermisi, M. and Dal Maso, G. and Fonseca, I. and Leoni, G.},
     TITLE = {Singular perturbation models in phase transitions for
              second-order materials},
   JOURNAL = {Indiana Univ. Math. J.},
  FJOURNAL = {Indiana University Mathematics Journal},
    VOLUME = {60},
      YEAR = {2011},
    NUMBER = {2},
     PAGES = {367--409},
      ISSN = {0022-2518,1943-5258},
   MRCLASS = {49J45 (35Q56 58E50 74K15 74N15 74Q05)},
  MRNUMBER = {2963779},
MRREVIEWER = {Erik\ J.\ Balder},
       DOI = {10.1512/iumj.2011.60.4346},
       URL = {https://doi.org/10.1512/iumj.2011.60.4346},
}

@article {ContiSchweizer2006nonlin,
    AUTHOR = {Conti, Sergio and Schweizer, Ben},
     TITLE = {Rigidity and gamma convergence for solid-solid phase
              transitions with {SO}(2) invariance},
   JOURNAL = {Comm. Pure Appl. Math.},
  FJOURNAL = {Communications on Pure and Applied Mathematics},
    VOLUME = {59},
      YEAR = {2006},
    NUMBER = {6},
     PAGES = {830--868},
      ISSN = {0010-3640,1097-0312},
   MRCLASS = {74G65 (49J45 74N15)},
  MRNUMBER = {2217607},
MRREVIEWER = {Giovanni\ Alberti},
       DOI = {10.1002/cpa.20115},
       URL = {https://doi.org/10.1002/cpa.20115},
}

@article {ContiSchweizer2006lin,
    AUTHOR = {Conti, Sergio and Schweizer, Ben},
     TITLE = {A sharp-interface limit for a two-well problem in
              geometrically linear elasticity},
   JOURNAL = {Arch. Ration. Mech. Anal.},
  FJOURNAL = {Archive for Rational Mechanics and Analysis},
    VOLUME = {179},
      YEAR = {2006},
    NUMBER = {3},
     PAGES = {413--452},
      ISSN = {0003-9527,1432-0673},
   MRCLASS = {74G65 (49J45 74N99)},
  MRNUMBER = {2208322},
MRREVIEWER = {Marc\ Oliver\ Rieger},
       DOI = {10.1007/s00205-005-0397-y},
       URL = {https://doi.org/10.1007/s00205-005-0397-y},
}

@incollection {ContiSchweizer2006nonlinimp,
    AUTHOR = {Conti, Sergio and Schweizer, Ben},
     TITLE = {Gamma convergence for phase transitions in impenetrable
              elastic materials},
 BOOKTITLE = {Multi scale problems and asymptotic analysis},
    SERIES = {GAKUTO Internat. Ser. Math. Sci. Appl.},
    VOLUME = {24},
     PAGES = {105--118},
 PUBLISHER = {Gakkotosho, Tokyo},
      YEAR = {2006},
      ISBN = {4-7625-0433-5},
   MRCLASS = {74N99 (49J45 74G65 74Q05)},
  MRNUMBER = {2233173},
MRREVIEWER = {Georg\ K.\ Dolzmann},
}

@article {DavoliFriedrich2020,
    AUTHOR = {Davoli, Elisa and Friedrich, Manuel},
     TITLE = {Two-well rigidity and multidimensional sharp-interface limits
              for solid-solid phase transitions},
   JOURNAL = {Calc. Var. Partial Differential Equations},
  FJOURNAL = {Calculus of Variations and Partial Differential Equations},
    VOLUME = {59},
      YEAR = {2020},
    NUMBER = {2},
     PAGES = {Paper No. 44, 47},
      ISSN = {0944-2669,1432-0835},
   MRCLASS = {74N15 (26B30 49J45 74G65)},
  MRNUMBER = {4062042},
MRREVIEWER = {Yasemin\ \c Seng\"ul},
       DOI = {10.1007/s00526-020-1699-5},
       URL = {https://doi.org/10.1007/s00526-020-1699-5},
}

@article {DavoliFriedrich2025,
    AUTHOR = {Davoli, Elisa and Friedrich, Manuel},
     TITLE = {Two-well linearization for solid-solid phase transitions},
   JOURNAL = {J. Eur. Math. Soc. (JEMS)},
  FJOURNAL = {Journal of the European Mathematical Society (JEMS)},
    VOLUME = {27},
      YEAR = {2025},
    NUMBER = {2},
     PAGES = {615--707},
      ISSN = {1435-9855,1435-9863},
   MRCLASS = {35Q74 (49J45 49Q20 74E99)},
  MRNUMBER = {4859576},
MRREVIEWER = {Xiaoguang\ Allan\ Zhong},
       DOI = {10.4171/jems/1385},
       URL = {https://doi.org/10.4171/jems/1385},
}

@article {StinsonKerrek2021,
    AUTHOR = {Stinson, Kerrek},
     TITLE = {On {$\Gamma$}-convergence of a variational model for
              lithium-ion batteries},
   JOURNAL = {Arch. Ration. Mech. Anal.},
  FJOURNAL = {Archive for Rational Mechanics and Analysis},
    VOLUME = {240},
      YEAR = {2021},
    NUMBER = {1},
     PAGES = {1--50},
      ISSN = {0003-9527,1432-0673},
   MRCLASS = {49J45 (78M30)},
  MRNUMBER = {4228856},
       DOI = {10.1007/s00205-020-01602-7},
       URL = {https://doi.org/10.1007/s00205-020-01602-7},
}

\end{document}